\documentclass[11pt]{article}
\usepackage[a4paper, total={7in, 9in}]{geometry}
\usepackage{amsmath}
\usepackage{amsthm}
\usepackage{amssymb}
\usepackage{cancel}
\usepackage[dvips]{graphicx}
\usepackage[utf8]{inputenc}
\usepackage{color}
\usepackage{bbm}
\usepackage{enumerate}
\usepackage{multicol}
\usepackage{fancyhdr}
\usepackage{graphicx}
\usepackage{lipsum}
\usepackage{booktabs}
\usepackage{subfig}
\usepackage{amsthm}
\usepackage{xcolor}
\usepackage[colorlinks=true]{hyperref}
\usepackage[ruled,vlined]{algorithm2e}
\usepackage{hyperref}
\usepackage[toc]{appendix}
\allowdisplaybreaks

\newcommand{\R}{{\mathbb R}}

\newcommand{\norm}[1]{\left\Vert#1\right\Vert}
\newcommand\inner[2]{\langle #1, #2 \rangle}

\newtheorem{theorem}{Theorem}[section]
\newtheorem{proposition}[theorem]{Proposition}
\newtheorem{lemma}[theorem]{Lemma}
\newtheorem{corollary}[theorem]{Corollary}
\newtheorem*{lemma*}{Lemma}
\theoremstyle{definition}
\newtheorem{remark}[theorem]{Remark}
\newtheorem{definition}[theorem]{Definition}
\newtheorem{example}[theorem]{Example}

\DeclareMathOperator{\dist}{dist}
\DeclareMathOperator{\argmin}{argmin}

\title{A speed restart scheme for a dynamics with Hessian driven damping}
\author{Juan José Maulén\footnote{Mathematical Engineering Department, Center for Mathematical Modeling (CNRS IRL2807), University of Chile,
		Santiago, Chile. Bernoulli Institute for Mathematics, Computer Science
		and Artificial Intelligence, Faculty of Science and Engineering,
		University of Groningen, Groningen, The Netherlands}, Juan Peypouquet\footnote{Bernoulli Institute for Mathematics, Computer Science
		and Artificial Intelligence, Faculty of Science and Engineering,
		University of Groningen, Groningen, The Netherlands} }

\begin{document}
	
	\maketitle
	\begin{abstract}
	In this paper, we analyze a speed restarting scheme for the dynamical system given by
	$$
	\ddot{x}(t) + \dfrac{\alpha}{t}\dot{x}(t) + \nabla \phi(x(t)) + \beta \nabla^2 \phi(x(t))\dot{x}(t)=0,   
	$$
	where $\alpha$ and $\beta$ are positive parameters, and $\phi:\R^n \to \R$ is a smooth convex function. If $\phi$ has quadratic growth, we establish a linear convergence rate for the function values along the restarted trajectories. As a byproduct, we improve the results obtained by Su, Boyd and Candès \cite{JMLR:v17:15-084}, obtained in the strongly convex case for $\alpha=3$ and $\beta=0$. Preliminary numerical experiments suggest that both adding a positive Hessian driven damping parameter $\beta$, and implementing the restart scheme help improve the performance of the dynamics and corresponding iterative algorithms as means to approximate minimizers of $\phi$. \\
	
	
	

\end{abstract}
	\section{Introduction}

In convex optimization, first-order methods are iterative algorithms that use function values and (generalized) derivatives to build minimizing sequences. Perhaps the oldest and simplest of them is the {\it gradient method} \cite{cauchy1847methode}, which can be interpreted as a finite-difference discretization of the differential equation
\begin{equation} \label{eq:grad}
	\dot x(t)+\nabla\phi\big(x(t)\big)=0,
\end{equation}
describing the {\it steepest descent dynamics}. The gradient method is applicable to smooth functions, but there are more contemporary variations that can deal with nonsmooth ones, and even exploit the functions' structure to enhance the algorithm's {\it per iteration} complexity, or overall performance. A keynote example (see \cite{Beck_book} for further insight) is the {\it proximal-gradient} (or {\it forward-backward}) method  
\cite{LionsMercier,Passty}, (see also \cite{martinet1970regularisation,rockafellar1976monotone}), which is in close relationship with a nonsmooth version of \eqref{eq:grad}. In any case, the analysis of related differential equations or inclusions is a valuable source of insight into the dynamic behavior of these iterative algorithms. \\

In \cite{polyakInertial,Nesterov1983AMF}, the authors introduced inertial substeps in the iterations of the gradient method, in order to accelerate its convergence. This variation improves the worst-case convergence rate from $\mathcal O(k^{-1})$ to $\mathcal O(k^{-2})$. In the strongly convex case, the constants that describe the linear convergence rate are also improved. This method  was extended to the proximal-gradient case in \cite{beck2009fast}, and to fixed point iterations \cite{krasnosel1955two,mann1953mean} (see, for example \cite{mainge2008convergence,lorenz2015inertial,attouch2019convergence,AttouchCabot2019,fierro2022inertial}, among others). In \cite{JMLR:v17:15-084}, Su, Boyd and Candès showed that Nesterov's inertial gradient algorithm$-$and, analogously, the proximal variant$-$can be interpreted as a discretization of the ordinary differential equation with Asymptotically Vanishing Damping
\begin{equation}\label{eq:avd_intr}\tag{AVD}
\ddot{x}(t) + \dfrac{\alpha}{t}\dot{x}(t) + \nabla \phi(x(t)) =0, 
\end{equation}
where $\alpha>0$. The function values vanish along the trajectories \cite{attouch2018fast}, and they do so at a rate of $\mathcal{O}(t^{-2})$ for $\alpha \geq 3$ \cite{JMLR:v17:15-084}, and $o(t^{-2})$ for $\alpha > 3$ \cite{may2017asymptotic}, in the worst-case scenario. \\


Despite its faster convergence rate guarantees, trajectories satisfying \eqref{eq:avd_intr}$-$as well as sequences generated by inertial first order methods$-$exhibit a somewhat chaotic behavior, especially if the objective function is ill-conditioned. In particular, the function values tend not to decrease monotonically, but to present an oscillatory character, instead.

\begin{example} \label{EG:AVD_R3}
We consider the quadratic function $\phi:\R^3\to\R$, defined by 
\begin{equation} \label{eq:Phi_R3}
	\phi(x_1,x_2,x_3)=\frac{1}{2}(x_1^2+\rho x_2^2+\rho^2x_3^2),
\end{equation}
whose condition number is $\max\{\rho^2,\rho^{-2}\}$. Figure \ref{fig:AVD} shows the behavior of the solution to \eqref{eq:avd_intr}, with $x(1)=(1,1,1)$ and $\dot x(1)=-\nabla \phi\big(x(1)\big)$ (the direction of maximum descent).
\end{example}

\begin{figure}[h]
\centering
\includegraphics[width=0.4\textwidth]{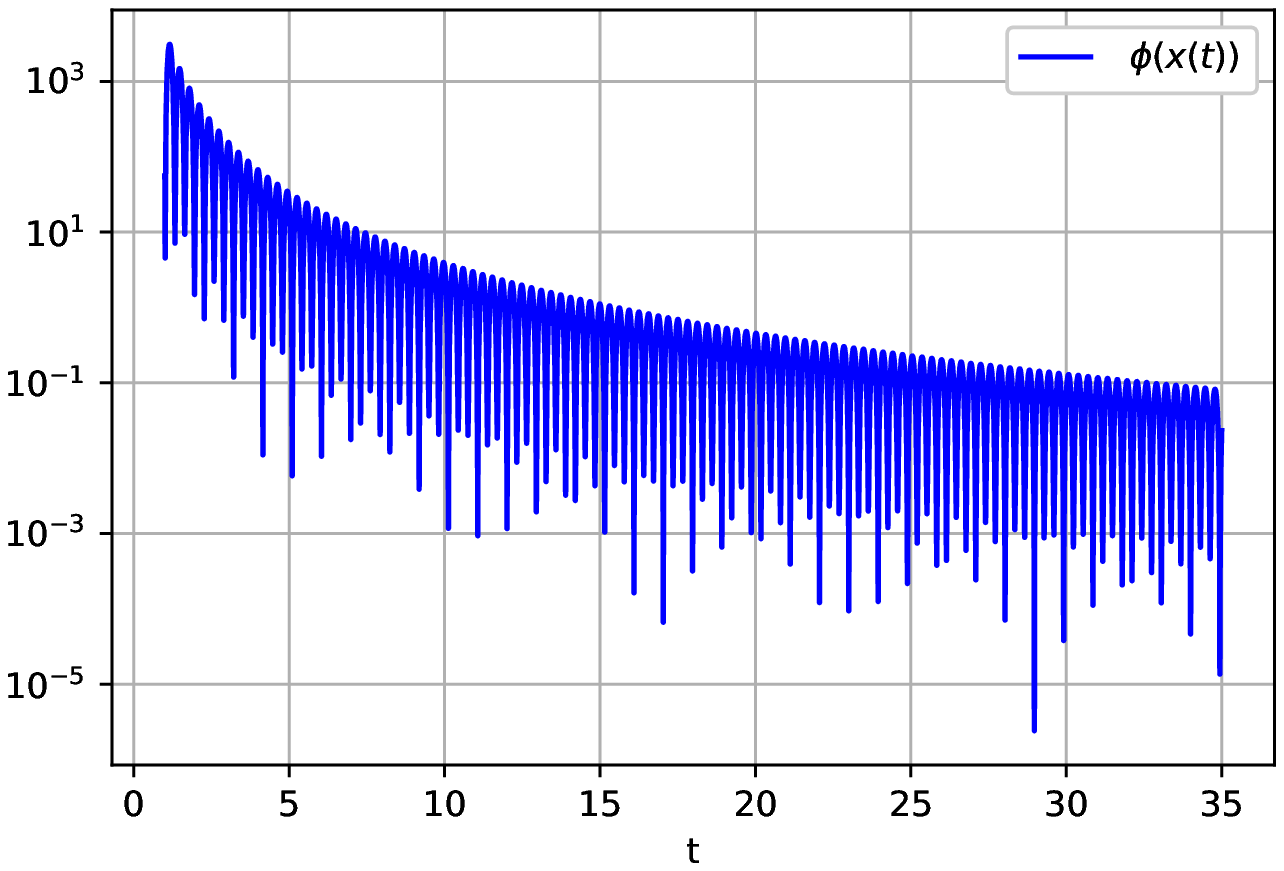}
\includegraphics[width=0.4\textwidth]{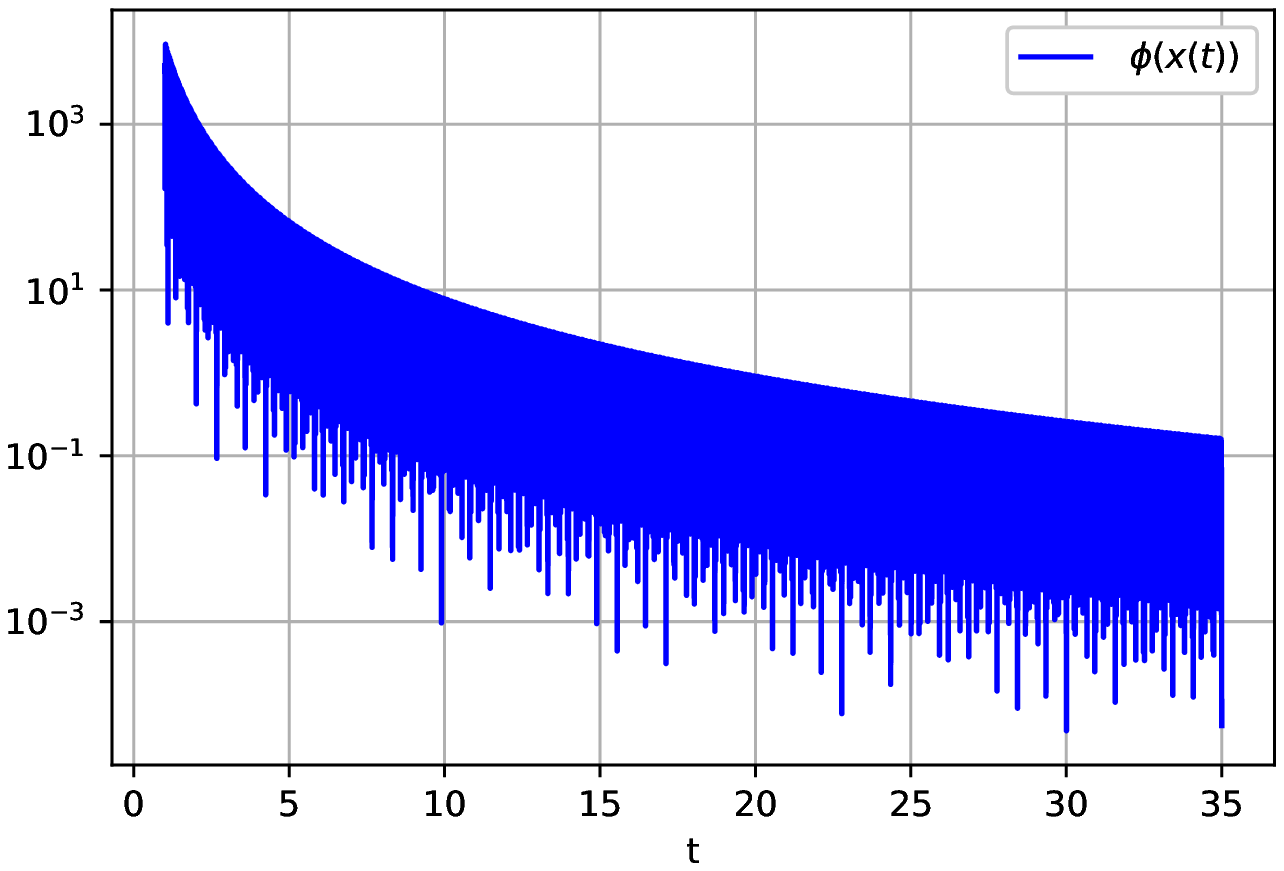}
\caption{Depiction of the function values according to Example \ref{EG:AVD_R3}, on the interval $[1,35]$, for $\alpha=3.1$, and $\rho=10$ (left) and $\rho=100$ (right).}
\label{fig:AVD}
\end{figure}

In order to avoid this undesirable behavior, and partly inspired by a continuous version of Newton's method \cite{alvarez1998dynamical},  Attouch, Peypouquet and Redont \cite{attouch2016fast} proposed a Dynamic Inertial Newton system with Asymptotically Vanishing Damping, given by
\begin{equation}\label{eq:din_avd}\tag{DIN-AVD}
\ddot{x}(t) + \dfrac{\alpha}{t}\dot{x}(t) + \nabla \phi(x(t)) + \beta \nabla^2 \phi(x(t))\dot{x}(t)=0, 
\end{equation}
where $\alpha,\beta>0$. In principle, this expression only makes sense when $\phi$ is twice differentiable, but the authors show that it can be transformed into an equivalent first-order equation in time and space, which can be extended to a differential inclusion that is well posed whenever $\phi$ is closed and convex.
The authors presented \eqref{eq:din_avd} as a continuous-time model for the design of new algorithms, a line of research already outlined in \cite{attouch2016fast}, and continued in \cite{attouch2020first}. Back to \eqref{eq:din_avd}, the function values vanish along the solutions, with the same rates as for \eqref{eq:avd_intr}. Nevertheless, in contrast with the solutions of \eqref{eq:avd_intr}, the oscillations are tame.

\begin{example} \label{EG:AVD_GIN_R3}
In the context of Example \ref{EG:AVD_R3}, Figure \ref{fig:AVD_GIN} shows the behavior of the solution to \eqref{eq:din_avd} in comparison with that of \eqref{eq:din_avd}, both with $x(1)=(1,1,1)$ and $\dot x(1)=-\nabla \phi\big(x(1)\big)$. 
\end{example}

\begin{figure}[h!]
\centering
\includegraphics[width=0.4\textwidth]{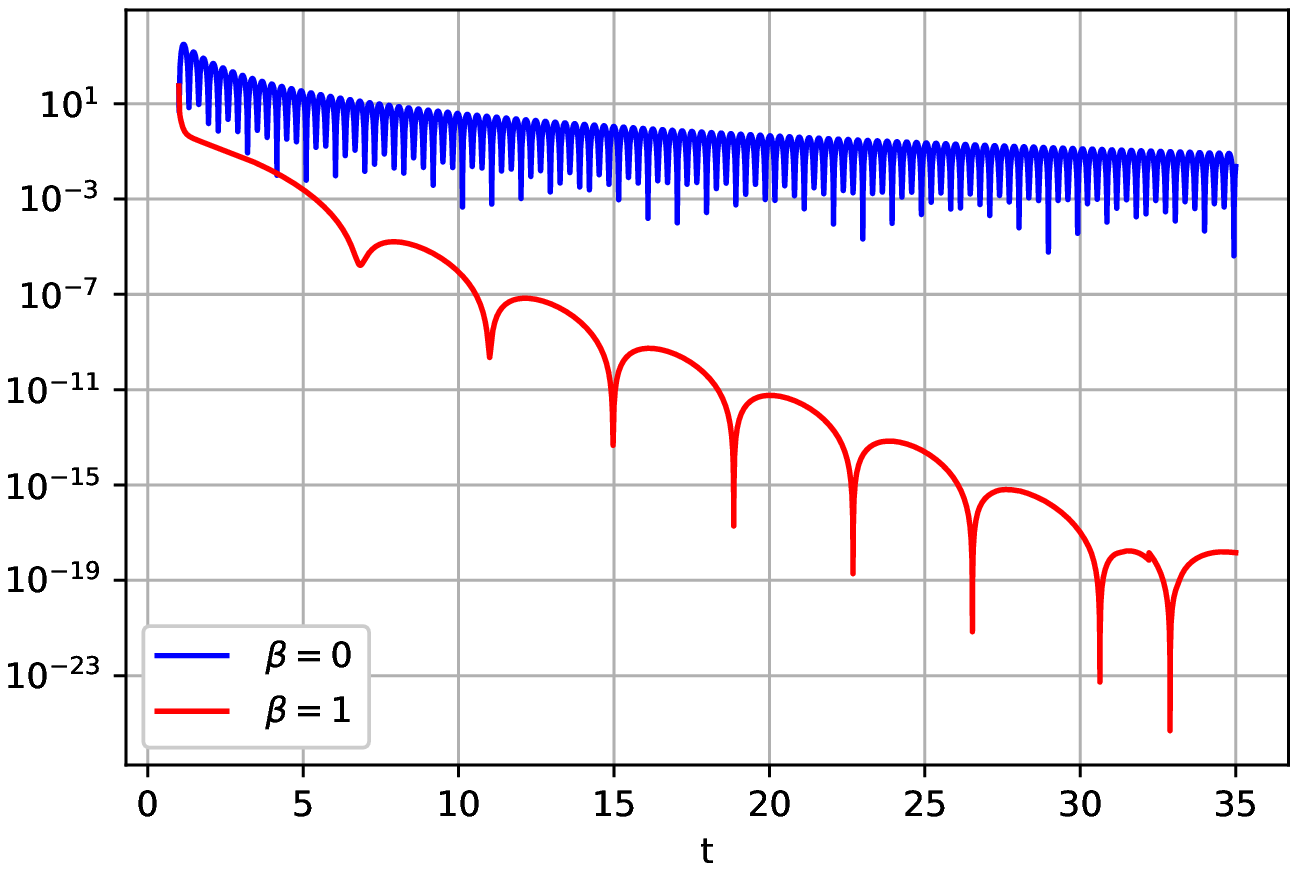}
\includegraphics[width=0.4\textwidth]{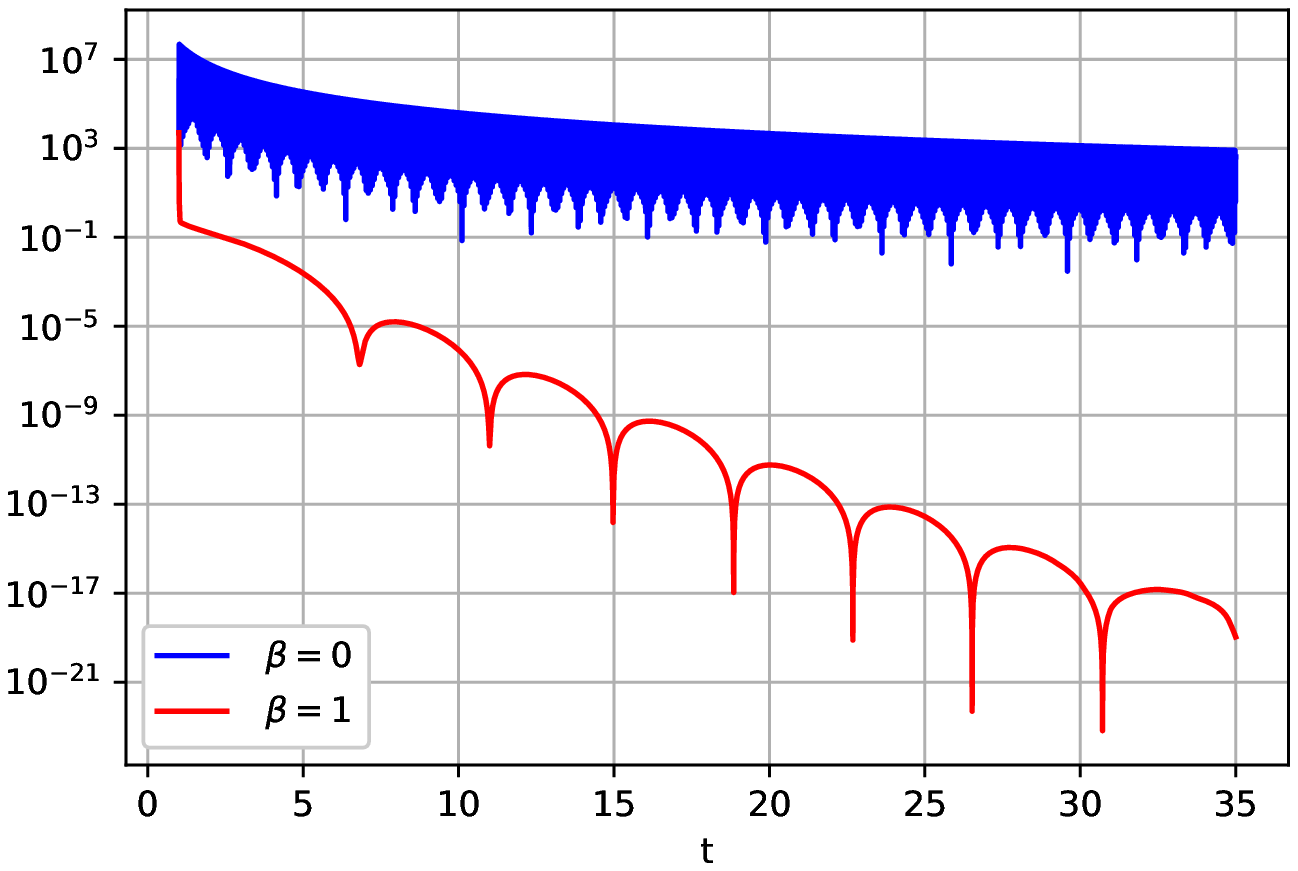}
\caption{Depiction of the function values according to Example \ref{EG:AVD_GIN_R3}, on the interval $[1,35]$, for $\alpha=3.1$, $\beta=1$, and $\rho=10$ (left) and $\rho=100$ (right).}
\label{fig:AVD_GIN}
\end{figure}

An alternative way to avoid$-$or at least moderate$-$the oscillations exemplified in Figure \ref{fig:AVD} for the solutions of \eqref{eq:avd_intr} is to stop the evolution and restart it with zero initial velocity, from time to time. The simplest option is to do so periodically, at fixed intervals. This idea is used in \cite{nesterov2013gradient} for the accelerated gradient method, where the number of iterations between restarts that depends on the parameter of strong convexity of the function. See also \cite{necoara2019linear,alamo2019restart,aujol:hal-03153525}, where the problem of estimating the appropriate restart times is addressed. An adaptive policy for the restarting of Nesterov's Method was proposed by O'Donoghue and Candès in \cite{o2015adaptive}, where the algorithm is restarted at the first iteration $k$ such that $\phi(x_{k+1})>\phi(x_{k})$, which prevents the function values to increase locally. This kind of restarting criteria shows a remarkable performance, although convergence rate guarantees have not been established, although some partial steps in this direction have been made in    \cite{giselsson2014monotonicity,lin_adaptive_2015}. Moreover, the authors of \cite{o2015adaptive} observe that this heuristic displays an erratic behavior when the difference $\phi(x_{k})-\phi(x_{k+1})$ is small, due to the prevalence of cancellation errors. Therefore, this method must be handled with care if high accuracy is desired. A different restarting scheme, based on
the speed of the trajectories, is proposed for \eqref{eq:avd_intr} in \cite{JMLR:v17:15-084}, where rates of convergence are established. The improvent can be remarkable, as shown in Figure \ref{fig:cont0}.

\begin{figure}[h]
\centering
\includegraphics[width=0.4\textwidth]{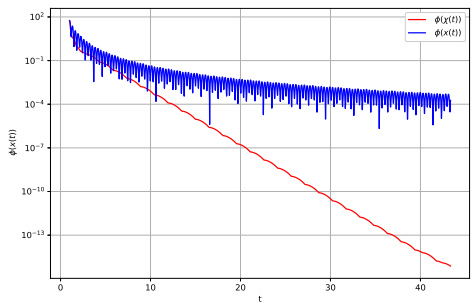}
\caption{Values along the trajectory, with (red) and without (blue) restarting, for \eqref{eq:avd_intr}. }
\label{fig:cont0}
\end{figure}

In \cite{JMLR:v17:15-084}, the authors also perform numerical tests using Nesterov's inertial gradient method, with this restarting scheme as a heuristic, and observe a faster convergence to the optimal value. \\

The aim of this work is to analyze the impact that the speed restarting scheme has on the solutions of \eqref{eq:din_avd}, in order to set the theoretical foundations to further accelerate Hessian driven inertial algorithms$-$like the ones in \cite{attouch2020first}$-$by means of a restarting policy. We provide linear convergence rates for functions with quadratic growth, and observe a noticeable improvement in the behavior of the trajectories in terms of stability and convergence speed, both in comparison with the non-restarted trajectories, and with the restarted solutions of \eqref{eq:avd_intr}. As a byproduct, we generalize improve some of the results in \cite{JMLR:v17:15-084}. \\

The paper is organized as follows: In section \ref{sec:restart_main}, we describe the speed restart scheme and state the convergence rate of the corresponding trajectories, which is the main theoretical result of this paper. Section \ref{sec:technicalities} contains the technical auxiliary results$-$especially some estimations on the restarting time$-$leading to the proof of our main result, which is carried out in Section \ref{sec:convergence_rate}. Finally, we present a few simple numerical examples in Section \ref{sec:illustration}, in order to illustrate the improvements, in terms of convergence speed, of the restarted trajectories.

	\section{Restarted trajectories for \eqref{eq:din_avd}}\label{sec:restart_main}

Throughout this paper, $\phi:\R^n\to \R$ is a twice continuously differentiable convex function, which attains its minimum value $\phi^*$, and whose gradient $\nabla \phi$ is Lipschitz-continuous with constant $L>0$. Consider the ordinary differential equation \eqref{eq:din_avd}, with initial conditions $x(0)=x_0$, $\dot{x}(0)=0$, and parameters $\alpha > 0$ and $\beta\ge 0$. A {\it solution} is a function in $\mathcal{C}^2\left( (0,+\infty); \R^n \right) \cap \mathcal{C}^1\left( [0,+\infty); \R^n \right)$, such that $x(0)=x_0$, $\dot{x}(0)=0$ and \eqref{eq:din_avd} holds for every $t>0$. Existence and uniqueness of such a solution is not straightforward due to the singularity at $t=0$, but can be established by a limiting procedure. As shown in Appendix \ref{Appendix}, we have the following:


\begin{theorem}\label{theo:existence}
	For every $x_0 \in \R^n$, the differential equation \eqref{eq:din_avd}, with initial conditions $x(0)=x_0$ and $\dot{x}(0)=0$, has a unique solution. 
\end{theorem}

We are concerned with the design and analysis of a restart scheme to accelerate the convergence of the solutions of \eqref{eq:din_avd} to minimizers of $\phi$, based on the method proposed in \cite{JMLR:v17:15-084}.  \\

\subsection{A speed restarting scheme and the main theoretical result}

Since the damping coefficient $\alpha/t$ goes to $0$ as $t\to\infty$, large values of $t$ result in a smaller stabilization of the trajectory. The idea is thus to restart the dynamics at the point where the speed ceases to increase. \\


Given $z\in\R^n$, let $y_z$ be the solution of \eqref{eq:din_avd}, with initial conditions $y_z(0)=z$ and $\dot{y}_z(0)=0$. Set
\begin{equation} \label{E:restart_time}
	T(z)=\inf\left\lbrace t >0\ :\ \dfrac{d}{dt} \norm{\dot{y}_z(t)}^2 \le 0  \right\rbrace.
\end{equation}

\begin{remark} \label{R:decreasing}
	Take $z\notin\argmin(\phi)$, and define $y_z$ as above. For $t\in\big(0,T(z)\big)$, we have
	$$
	\dfrac{d}{dt}\phi(y_{z}(t))=\inner{\nabla \phi (y_{z}(t))}{\dot{y}_{z}(t)} = - \inner{\ddot{y}_{z}(t)}{\dot{y}_{z}(t)} - \dfrac{\alpha}{t}\norm{\dot{y}_{z}(t)}^2 - \beta\inner{\nabla^2 \phi(y_{z}(t))\dot{y}_{z}(t)}{\dot{y}_{z}(t)}.
	$$
	But $\inner{\nabla^2 \phi(y_{z}(t))\dot{y}_{z}(t)}{\dot{y}_{z}(t)}\geq 0$ by convexity, and $\inner{\ddot{y}_{z}(t)}{\dot{y}_{z}(t)}\ge 0$ by the definition of $T(z)$. Therefore,
	\begin{equation}\label{eq:phi_derivative_bound}
		\dfrac{d}{dt}\phi(y_{z}(t))\leq  - \dfrac{\alpha}{t}\norm{\dot{y}_{z}(t)}^2.    
	\end{equation}
	In particular, $t\mapsto\phi\big(y_z(t)\big)$ decreases on $[0,T(z)]$.
\end{remark}

If $z\notin\argmin(\phi)$, then $T(z)$ cannot be $0$. In fact, we shall prove (see Corollaries \ref{C:inf} and \ref{C:sup}) that 
\begin{equation} \label{E:T(z)}
	0<\inf\big\{T(z):z\notin\argmin(\phi)\big\}\le \sup\big\{T(z):z\notin\argmin(\phi)\big\}<\infty.
\end{equation}

\begin{definition} \label{def:restarted}
	Given $x_0\in \R^n$, the {\it restarted trajectory} $\chi_{x_0}:[0,\infty)\to\R^n$ is defined inductively:
	\begin{enumerate}
		\item First, compute $y_{x_0}$, $T_1=T(x_0)$ and $S_1=T_1$, and define $\chi_{x_0}(t)=y_{x_0}(t)$ for $t\in[0,S_1]$.
		\item For $i\ge 1$, having defined $\chi_{x_0}(t)$ for $t\in[0,S_i]$, set $x_i=\chi_{x_0}(S_i)$, and compute $y_{x_i}$. Then, set $T_{i+1}=T(x_i)$ and $S_{i+1}=S_{i}+T_{i+1}$, and define $\chi_{x_0}(t)=y_{x_i}(t-S_{i})$ for $t\in(S_i,S_{i+1}]$.
	\end{enumerate}
\end{definition}

In view of \eqref{E:T(z)}, $S_i$ is defined for all $i\ge 1$, $\inf_{i\ge 1}(S_{i+1}-S_i)>0$ and $\lim_{i\to\infty}S_i=\infty$. Moreover, in view of Remark \ref{R:decreasing}, we have

\begin{proposition} \label{P:nonincreasing}
	The function $t\mapsto\phi\big(\chi_{x_0}(t)\big)$ is nonincreasing on $[0,\infty)$.
\end{proposition}

Our main theoretical result establishes that  $\phi\big(\chi_{x_0}(t)\big)$ converges linearly to $\phi^*$, provided there exists $\mu>0$ such that \begin{equation} \label{E:Loja}
	\mu(\phi(z) - \phi^*) \leq \dfrac{1}{2}\norm{\nabla \phi(z)}^2
\end{equation} 
for all $z\in\R^n$. The \L ojasiewicz inequality \eqref{E:Loja} is equivalent to quadratic growth  and is implied by strong convexity (see \cite{bolte2017error}). More precisely, we have the following:



\begin{theorem} \label{teo:convergence_rate} Let $\phi:\R^n\to \R$ be convex and twice continuously differentiable. Assume $\nabla \phi$ is Lipschitz-continuous with constant $L>0$, there exists $\mu>0$ such that \eqref{E:Loja} holds, and that the minimum value $\phi^*$ of $\phi$ is attained. Given $\alpha\geq 3$ and $\beta>0$, let be the restarted trajectory defined by \eqref{eq:din_avd} from an initial point $x_0 \in \R^n$. Then, there exist constants $C,K>0$ such that
	$$\phi\big(\chi_{x_0}(t)\big)-\phi^*\le Ce^{-Kt}\big(\phi(x_0)-\phi^*\big)\le \frac{CL}{2}e^{-Kt}\dist\big(x_0,\argmin(\phi)\big)^2$$
	for all $t>0$.
\end{theorem}


The rather technical proof is split into several parts and presented in the next subsections.

\section{Technicalities} \label{sec:technicalities}

Throughout this section, we fix $z\notin\argmin(\phi)$ and, in order to simplify the notation, we denote by $x$ (instead of $y_{z}$) the solution of \eqref{eq:din_avd} with initial condition $x(0)=z$ and $\dot x(0)=0$.

\subsection{A few useful artifacts}

We begin by defining some useful auxiliary functions and point out the main relationships between them. \\

To this end, we first rewrite equation \eqref{eq:din_avd} as
\begin{equation}\label{eq:derivative_form}
	\dfrac{d}{dt}(t^\alpha\dot{x}(t))  = - t^\alpha\nabla \phi(x(t))  - \beta t^\alpha\nabla^2 \phi(x(t))\dot{x}(t).  \end{equation}

Integrating \eqref{eq:derivative_form} over $[0,t]$, we get 
\begin{align} \label{eq:eqIJ} 
	t^\alpha\dot{x}(t) & = - \int_{0}^{t}u^\alpha \nabla \phi(x(u))\, du -\beta   \int_{0}^{t}u^\alpha\nabla^2 \phi(x(u))\dot{x}(u)\, du \nonumber \\
	& = -\left[\int_{0}^{t} u^\alpha(\nabla \phi(x(u))- \nabla \phi(z))\, du\right] - \left[\beta   \int_{0}^{t}u^\alpha\nabla^2 \phi(x(u))\dot{x}(u)\, du\right] - \dfrac{t^{\alpha+1}}{\alpha+1}\nabla \phi(z). 
\end{align}
In order to obtain an upper bound for the speed $\dot x$, the integrals 
\begin{equation} \label{E:I_and_J}
	I_z(t)=\int_{0}^{t} u^\alpha(\nabla \phi(x(u))- \nabla \phi(z))\, du \qquad\hbox{and}\qquad J_z(t)=\beta   \int_{0}^{t}u^\alpha\nabla^2 \phi(x(u))\dot{x}(u)\, du
\end{equation}
will be majorized using the function 
\begin{equation} \label{E:M}
	M_z(t)=\sup_{u \in (0,t]} \left[ \dfrac{\norm{\dot{x}(u)}}{u} \right],
\end{equation}
which is positive, nondecreasing and continuous.

\begin{lemma} \label{L:I_and_J}
	For every $t>0$, we have
	$$\norm{I_z(t)}  \leq \dfrac{LM_z(t)t^{\alpha+3}}{2(\alpha+3)} \qquad\hbox{and}\qquad
	\norm{J_z(t)}  \le \dfrac{\beta L M_z(t)t^{\alpha+2}}{\alpha+2}.$$ 
\end{lemma}

\begin{proof}
	For the first estimation, we use the Lipschitz-continuity of $\nabla \phi$ and the fact that $M$ in nondecreasing, to obtain
	$$\norm{\nabla \phi(x(u)) - \nabla \phi(z)} \leq L\|x(u)-z\| \leq L \norm{\int_{0}^{u} \dot{x}(s)\, ds} \leq L \int_{0}^{u} s \dfrac{\norm{\dot{x}(s)}}{s}\, ds \leq LM_z(u) \int_{0}^{u} s \, ds ,$$
	which results in 
	\begin{equation} \label{E:bound_gradients}
		\norm{\nabla \phi(x(u)) - \nabla \phi(z)} \leq \dfrac{Lu^2M_z(u)}{2}
	\end{equation}
	Then, from the definition of $I_z(t)$ we deduce that
	$$\norm{I_z(t)} \leq \int_{0}^{t} u^\alpha \norm{\nabla\phi(x(u)) - \nabla\phi(z)} \, du
	\leq \dfrac{LM_z(t)}{2} \int_{0}^{t} u^{\alpha+2} \, du = \dfrac{LM_z(t)t^{\alpha+3}}{2(\alpha+3)}.$$
	For the second inequality, we proceed analogously to get  
	$$ \norm{\nabla^2 \phi(x(u))\dot x(u)} 
	= \norm{{\lim_{r \to u}} \dfrac{\nabla \phi(x(r)) - \nabla \phi(x(u))}{r-u}} 
	\leq {\lim_{r \to u}} \dfrac{L}{r-u} \int_{u}^{r}\norm{\dot{x}(s)}\, ds 
	\leq {\lim_{r \to u}} \dfrac{L M_z(r)}{r-u} \int_{u}^{r} s\, ds,$$
	which yields
	\begin{equation} \label{E:bound_Hessian}
		\norm{\nabla^2 \phi(x(u))\dot x(u)} \le  LuM_z(u).
	\end{equation}
	Then,    
	$$\norm{J_z(t)} \leq \beta   \int_{0}^{t}u^\alpha \norm{\nabla^2\phi(x(u))\dot{x}(u)}\, du \leq \beta \int_{0}^{t} u^{\alpha+1} LM_z(u) \,du \leq \dfrac{\beta L M_z(t)t^{\alpha+2}}{\alpha+2},$$
	as claimed.
\end{proof}

The dependence of $M_z$ on the initial condition $z$ may be greatly simplified. To this end, set
\begin{equation} \label{E:H}
	H(t)=1 -\dfrac{L\beta t}{(\alpha+2)} - \dfrac{Lt^2}{2(\alpha+3)}.
\end{equation}
The function $H$ is concave, quadratic, does not depend on $z$, and has exactly one positive zero, given by
\begin{equation} \label{E:tau1}
	\tau_1=- \left(\dfrac{ \alpha+3}{\alpha+2}\right)\beta+\sqrt{\left(\dfrac{ \alpha+3}{\alpha+2}\right)^2\beta^2 + \dfrac{2(\alpha+3)}{L}}.
\end{equation}
In particular, $H$ decreases strictly from $1$ to $0$ on $[0,\tau_1]$. 

\begin{lemma} \label{lem:lemma1}
	For every $t \in (0,\tau_1)$, 
	\begin{equation} \label{eq:bound_M}
		M_z(t) \leq \dfrac{\norm{\nabla \phi (z)}}{(\alpha+1)H(t)}.    
	\end{equation}
\end{lemma}

\begin{proof} 
	If $0< u\le t$, using \eqref{eq:eqIJ} and \eqref{E:I_and_J}, along with Lemma \ref{L:I_and_J}, we obtain
	\begin{equation} \label{eq:bound_dotx}
		\dfrac{\norm{\dot{x}(u)}}{u} \leq \dfrac{\norm{I_z(u)+J_z(u)}}{u^{\alpha+1}}  + \dfrac{\norm{\nabla \phi(z)}}{\alpha+1}
		\leq \left[\dfrac{L u^2}{2(\alpha+3)} +  \dfrac{L\beta u}{\alpha+2} \right] M_z(u) + \dfrac{\norm{\nabla \phi(z)}}{\alpha+1}.
	\end{equation}
	Since the right-hand side is nondecreasing in $t$, we take the supremum for $u\in[0,t]$ to deduce that
	$$M_z(t) \leq \left[\dfrac{L t^2}{2(\alpha+3)} +  \dfrac{L\beta t}{\alpha+2} \right] M_z(t) + \dfrac{\norm{\nabla \phi(z)}}{\alpha+1}.$$
	Rearranging the terms, and using the definition of $H$, given in \eqref{E:H}, we see that
	$$H(t)M_z(t) \leq  \dfrac{\norm{\nabla \phi (z)}}{(\alpha+1)}.$$
	We conclude by observing that $H$ is positive on $(0,\tau_1)$. 
\end{proof}

By combining Lemmas \ref{L:I_and_J} and \ref{lem:lemma1}, and inequalities \eqref{E:bound_gradients} and \eqref{E:bound_Hessian}, we obtain:

\begin{corollary} \label{C:I_and_J}
	For every $t \in (0,\tau_1)$, we have
	\begin{align*}
		\norm{I_z(t)+J_z(t)} 
		\if{
			\le 
			t^{\alpha+1}\,\left[\dfrac{L t^2}{2(\alpha+3)} +  \dfrac{L\beta t}{\alpha+2} \right]\dfrac{\norm{\nabla \phi (z)}}{(\alpha+1)H(t)} =
		}\fi 
		& \le  t^{\alpha+1}\,\left[\dfrac{1-H(t)}{H(t)} \right]\dfrac{\norm{\nabla \phi (z)}}{(\alpha+1)} \\
		\left\|\big(\nabla \phi (x(t)) - \nabla \phi (z)\big) + \beta \nabla^2\phi(x(t))\dot{x}(t)\right\| & \le \left[\dfrac{Lt^2}{2}+\beta Lt\right] \dfrac{\norm{\nabla \phi (z)}}{(\alpha+1)H(t)}.
	\end{align*}
\end{corollary}

We highlight the fact that the bound above depends on $z$ only via the factor $\norm{\nabla \phi (z)}$.

\subsection{Estimates for the restarting time}

We begin by finding a lower bound for the restarting time, depending on the parameters $\alpha$, $\beta$ and $L$, but not on the initial condition $z$.

\begin{lemma} \label{lem:technical}
	Let $z\notin\argmin(\phi)$, and let $x$ be the solution of \eqref{eq:din_avd} with initial conditions $x(0)=z$ and $\dot x(0)=0$. For every $t\in(0,\tau_1)$, we have
	$$
	\inner{\dot{x}(t)}{\ddot{x}(t)}\ge 
	\frac{t\,\|\nabla \phi(z)\|^2}{(\alpha+1)^2H(t)^2}\left(1-\dfrac{(2\alpha+3)\beta Lt}{(\alpha+2)}-\dfrac{(\alpha+2)Lt^2}{(\alpha+3)}\right).
	$$
\end{lemma}

\begin{proof}
	From \eqref{eq:eqIJ} and \eqref{E:I_and_J}, we know that 
	\begin{equation}\label{eq:dotx}
		\dot{x}(t)= -\dfrac{1}{t^\alpha}\big(I_z(t) +J_z(t) \big) -  \dfrac{t}{\alpha+1}\nabla \phi (z).   
	\end{equation}
	On the other hand,
	$$\dfrac{d}{dt}\left[ \dfrac{1}{t^\alpha}\big(I_z(t) +J_z(t) \big) \right] = -\dfrac{\alpha}{t^{\alpha+1}}\big(I_z(t) +J_z(t) \big) +\left(\nabla \phi \big(x(t)\big) - \nabla \phi (z)\right) + \beta\nabla^2 \phi(x(t))\dot{x}(t).$$
	Then, 
	$$\ddot{x}(t)=  \dfrac{\alpha}{t^{\alpha+1}} \big(I_z(t)+J_z(t)\big)  -(\nabla \phi (x(t)) - \nabla \phi (z)) - \beta \nabla^2\phi(x(t))\dot{x}(t) - \dfrac{1}{\alpha+1}\nabla\phi(z)= A(t)-B(t),$$
	where 
	$$A(t)=\dfrac{\alpha}{t^{\alpha+1}} \big(I_z(t)+J_z(t)\big)- \dfrac{1}{\alpha+1}\nabla\phi(z) \qquad\hbox{and}\qquad B(t)=\big(\nabla \phi (x(t)) - \nabla \phi (z)\big) + \beta \nabla^2\phi(x(t))\dot{x}(t) $$
	With this notation, we have
	$$
	\inner{\dot x(t)}{\ddot x(t)} = \inner{\dot x(t)}{A(t)}-\inner{\dot x(t)}{B(t)}\ge \inner{\dot x(t)}{A(t)}-\|\dot x(t)\|\,\|B(t)\|.$$
	For the first term, we do as follows:
	\begin{align*}
		\inner{\dot x(t)}{A(t)} & = -\left\langle \dfrac{1}{t^\alpha}\big(I_z(t) +J_z(t) \big) + \dfrac{t}{\alpha+1}\nabla \phi (z), \dfrac{\alpha}{t^{\alpha+1}} \big(I_z(t)+J_z(t)\big)- \dfrac{1}{\alpha+1}\nabla\phi(z) \right\rangle \\ 
		& \ge \frac{t}{(\alpha+1)^2}\|\nabla \phi(z)\|^2 -\dfrac{\alpha}{t^{2\alpha+1}}\|I_z(t)+J_z(t)\|^2 -\dfrac{(\alpha-1)}{t^{\alpha}(\alpha+1)}\|\nabla \phi(z)\|\,\|I_z(t)+J_z(t)\| \\
		& \ge \frac{t}{(\alpha+1)^2}\|\nabla \phi(z)\|^2 -\dfrac{\alpha t}{(\alpha+1)^2}\left[\dfrac{1-H(t)}{H(t)} \right]^2\norm{\nabla \phi (z)}^2 -\dfrac{(\alpha-1)t}{(\alpha+1)^2}\left[\dfrac{1-H(t)}{H(t)} \right]\|\nabla \phi(z)\|^2 \\
		& = \frac{t\,\|\nabla \phi(z)\|^2}{(\alpha+1)^2}\left(1-\alpha\left[\dfrac{1-H(t)}{H(t)} \right]^2 - (\alpha-1)\left[\dfrac{1-H(t)}{H(t)} \right]\right)\\
		& = \frac{t\,\|\nabla \phi(z)\|^2}{(\alpha+1)^2H(t)^2}\left(H(t)^2-\alpha\big(1-H(t)\big)^2 - (\alpha-1)H(t)\big(1-H(t)\big)\right)\\
		& = \frac{t\,\|\nabla \phi(z)\|^2}{(\alpha+1)^2H(t)^2}\big((\alpha+1)H(t)-\alpha\big),
	\end{align*}
	where we have used the Cauchy-Schwarz inequality and Corollary \ref{C:I_and_J}. For the second term, we first use \eqref{eq:dotx} and observe that 
	$$\|\dot x(t)\|\le \dfrac{1}{t^\alpha}\|I_z(t)+J_z(t)\|+\dfrac{t}{(\alpha+1)}\|\nabla \phi(z)\|\le \dfrac{t\,\|\nabla \phi(z)\|}{(\alpha+1)H(t)},$$
	and
	$$B(t)\le \left[\dfrac{Lt^2}{2}+\beta Lt\right] \dfrac{\|\nabla \phi (z)\|}{(\alpha+1)H(t)},$$
	by Corollary \ref{C:I_and_J}. We conclude that
	\begin{align*}
		\inner{\dot x(t)}{\ddot x(t)} 
		& \ge \frac{t\,\|\nabla \phi(z)\|^2}{(\alpha+1)^2H(t)^2}\left((\alpha+1)H(t)-\alpha-\dfrac{Lt^2}{2}-\beta Lt\right) \\
		& = \frac{t\,\|\nabla \phi(z)\|^2}{(\alpha+1)^2H(t)^2}\left(1-\dfrac{(2\alpha+3)\beta Lt}{(\alpha+2)}-\dfrac{(\alpha+2)Lt^2}{(\alpha+3)}\right),
	\end{align*}
	as stated.
\end{proof}

The function $G$, defined by 
\begin{equation} \label{E:technical}
	G(t)=1-\dfrac{(2\alpha+3)\beta Lt}{(\alpha+2)}-\dfrac{(\alpha+2)Lt^2}{(\alpha+3)}=(\alpha+1)H(t)-\alpha-\dfrac{Lt^2}{2}-\beta Lt,
\end{equation}
does not depend on the initial condition $z$. Its unique positive zero is 
\begin{equation}\label{eq:tau_3}
	\tau_3=- \dfrac{(\alpha+3)(2\alpha+3)}{2(\alpha+2)^2}\beta+\sqrt{\dfrac{(\alpha+3)^2(2\alpha+3)^2}{4(\alpha+2)^4}\beta^2 + \dfrac{(\alpha+3)}{(\alpha+2)L}}.    
\end{equation}

In view of the definition of the restarting time, an immediate consequence of Lemma \ref{lem:technical} is

\begin{corollary} \label{C:inf}
	Let $T_*=\inf\big\{T(z):z\notin\argmin(\phi)\big\}$. Then, $\tau_3\le T_*$.
\end{corollary}





\begin{remark} \label{R:tau_3} If $\beta=0$, then
	$$\tau_3=\sqrt{\dfrac{(\alpha+3)}{(\alpha+2)L}}.$$
	The case $\alpha=3$ and $\beta=0$ was studied in \cite{JMLR:v17:15-084}, and the authors provided $\frac{4}{5\sqrt{L}}$ as a lower bound for the restart. The arguments presented here yield a higher bound, since
	$$\tau_3=\sqrt{\dfrac{6}{5L}} > \dfrac{1}{\sqrt{L}} >
	\frac{4}{5\sqrt{L}}.$$
\end{remark}


Recall that the function $H$ given in \eqref{E:H} decreases from $1$ to $0$ on $[0,\tau_1]$. Therefore, $H(t)>\frac{1}{2}$ for all $t\in [0,\tau_2)$, where
\begin{equation} \label{E:tau2}
	\tau_2=H^{-1}\big(\hbox{$\frac{1}{2}$}\big)=-\left(\dfrac{ \alpha+3}{\alpha+2}\right)\beta + \sqrt{\left( \dfrac{\alpha+3}{\alpha+2} \right)^2 \beta^2 + \dfrac{\alpha+3}{L}}<\tau_1. 
\end{equation}
Evaluating the right-hand side of \eqref{E:technical}, we see that
$$G(\tau_2)=\frac{(1-\alpha)-L\tau_2^2-2\beta L\tau_2}{2}<0,$$
whence
\begin{equation} \label{E:taus}
	\tau_1>\tau_2>\tau_3>0.
\end{equation}
These facts will be useful to provide an upper bound for the restarting time.

\begin{proposition}\label{P:sup} 
	Let $z\notin\argmin(\phi)$, and let $x$ be the solution of \eqref{eq:din_avd} with initial conditions $x(0)=z$ and $\dot x(0)=0$. Let $\phi$ satisfy \eqref{E:Loja} with $\mu>0$. For each $\tau\in(0,\tau_2)\cap(0,T(z)]$, we have
	$$T(z) \leq \tau \exp \left[ \dfrac{(\alpha+1)^2}{2\alpha\mu\tau^2\Psi(\tau)} \right], \qquad\hbox{where}\qquad
	\Psi(\tau)=\left[2-\frac{1}{H(\tau)}\right]^2.$$
\end{proposition}

\begin{proof}
	In view of \eqref{eq:eqIJ} and \eqref{E:I_and_J}, we can use Corollary \ref{C:I_and_J} to obtain
	$$\norm{\dot{x}(\tau) + \frac{\tau}{\alpha+1}\nabla \phi(z)} = \dfrac{1}{\tau^\alpha} \norm{ I(\tau) + J(\tau)  } \le  \tau\,\left[\dfrac{1}{H(\tau)}-1 \right]\dfrac{\norm{\nabla \phi (z)}}{(\alpha+1)}.$$
	From the (reverse) triangle inequality and the definition of $H$, it ensues that
	\begin{equation} \label{E:speed}
		\norm{\dot{x}(\tau)} \ge \dfrac{\tau\norm{\nabla \phi(z)}}{\alpha+1}-\tau\,\left[\dfrac{1}{H(\tau)}-1 \right]\dfrac{\norm{\nabla \phi (z)}}{(\alpha+1)} =\tau\left[2-\frac{1}{H(\tau)}\right]\dfrac{\norm{\nabla \phi(z)}}{\alpha+1},
	\end{equation}
	which is positive, because $\tau\in(0,\tau_2)$. Now, take $t \in [\tau,T(z)]$. Since $\norm{\dot{x}(t)}^2$ increases on $\big[0,T(z)\big]$, Remark \ref{R:decreasing} gives
	$$\dfrac{d}{dt}\phi\big(x(t)\big) \leq -\dfrac{\alpha}{t}\norm{\dot{x}(t)}^2 \leq -\dfrac{\alpha}{t}\norm{\dot{x}(\tau)}^2 \le -\dfrac{1}{t}\left[\dfrac{\alpha\tau^2\Psi(\tau)\norm{\nabla \phi(z)}^2}{(\alpha+1)^2}\right].  $$
	Integrating over $[\tau,T(z)]$, we get
	\begin{equation} \label{E:phi_log}
		\phi\big(x(T(z))\big) - \phi\big(x(\tau)\big) \leq -\left[\dfrac{\alpha\tau^2\Psi(\tau)\norm{\nabla \phi(z)}^2}{(\alpha+1)^2}\right]\ln\left[ \dfrac{T(z)}{\tau} \right].
	\end{equation}
	It follows that
	$$\left[\dfrac{\alpha\tau^2\Psi(\tau)\norm{\nabla \phi(z)}^2}{(\alpha+1)^2}\right]\ln\left[ \dfrac{T(z)}{\tau} \right]\le \phi\big(x(\tau)\big)- \phi\big(x(T(z))\big) \le \phi(z)-\phi^* \le \dfrac{\norm{\nabla \phi(z)}^2}{2\mu},$$
	in view of \eqref{E:Loja}. It suffices to rearrange the terms to conclude. 
\end{proof}

\begin{corollary} \label{C:sup}
	Let $\phi$ satisfy \eqref{E:Loja} with $\mu>0$, and let $\tau_*\in(0,\tau_2)\cap(0,T_*]$. Then,
	$$\sup\big\{T(z):z\notin\argmin(\phi)\big\} \leq \tau_* \exp \left[ \dfrac{(\alpha+1)^2}{2\alpha\mu\tau_*^2\Psi(\tau_*)} \right].$$
\end{corollary}

\section{Function value decrease and proof of Theorem \ref{teo:convergence_rate}} \label{sec:convergence_rate}

The next result provides the ratio at which the function values have been reduced by the time the trajectory is restarted.

\begin{proposition} \label{P:reduction} 
	Let $z\notin\argmin(\phi)$, and let $x$ be the solution of \eqref{eq:din_avd} with initial conditions $x(0)=z$ and $\dot x(0)=0$. Let $\phi$ satisfy \eqref{E:Loja} with $\mu>0$. For each $\tau\in(0,\tau_2)\cap(0,T(z)]$, we have
	\begin{equation*}\label{eq:linear_convergence}
		\phi\big(x(t)\big) - \phi^* \leq \left[1-\dfrac{\alpha\mu\tau^2\Psi(\tau)}{(\alpha+1)^2}\right]\big(\phi(z) - \phi^*\big)
	\end{equation*}
	for every $t\in[\tau,T(z)]$.
\end{proposition}

\begin{proof} 
	Take $s\in(0,\tau)$. By combining Remark \ref{R:decreasing} with \eqref{E:speed}, we obtain
	$$\dfrac{d}{ds}\phi(x(s)) \le - \dfrac{\alpha }{s}\norm{\dot{x}(s)}^2 \le - \dfrac{\alpha s\norm{\nabla \phi(z)}^2}{(\alpha+1)^2}\left[2-\frac{1}{H(s)}\right]^2 \le - \dfrac{\alpha s\norm{\nabla \phi(z)}^2}{(\alpha+1)^2}\Psi(\tau)   
	$$ 
	because $H$ decreases on $(0,\tau_1)$, which contains $(0,\tau)$. Integrating on $(0,\tau)$ and using \eqref{E:Loja}, we obtain
	$$\phi\big(x(\tau)\big)-\phi^*\le \phi(z)-\phi^*-\dfrac{\alpha \tau^2 \Psi(\tau) \norm{\nabla \phi(z)}^2}{2(\alpha+1)^2}\le \left[1-\dfrac{\alpha\mu \tau^2\Psi(\tau)}{(\alpha+1)^2}\right]\big(\phi(z)-\phi^*\big).$$
	To conclude, it suffices to observe that $\phi\big(x(t)\big)\le\phi\big(x(\tau)\big)$ in view of Remark \ref{R:decreasing}.
\end{proof}

\begin{remark}\label{r:bound} 
	Since $\Psi$ is decreasing in $[0,\tau_2)$, we have $\Psi(t)\ge\Psi(\tau_*)>0$, whenever $0\le t\le \tau_* < \tau_2$. Moreover, in view of \eqref{E:taus} and Corollary \ref{C:inf}, we can take $\tau_*=\tau_3$ to obtain a lower bound. If $\beta=0$, we obtain
	$$\Psi(t)\ge\Psi(\tau_3)=\left[2-\frac{1}{H(\tau_3)}\right]^2=\left[2-\frac{1}{1-\frac{1}{2(\alpha+2)}}\right]^2=\left[2-\frac{2\alpha+4}{2\alpha+3}\right]^2=\left[\dfrac{2\alpha+2}{2\alpha+3}\right]^2,$$
	which is independent of $L$. As a consequence, the inequality in Proposition \ref{P:reduction} becomes
	$$
	\phi\big(x(t)\big)-\phi^*\le \left( 1 - \frac{4\alpha(\alpha+3)}{(\alpha+2)(2\alpha+3)^2}\dfrac{\mu}{L}\right)(\phi(x_0) - \phi^*).    
	$$
	For $\alpha=3$, this gives
	$$
	\phi\big(x(t)\big)-\phi^*\le \left( 1 - \frac{8}{45}\dfrac{\mu}{L}\right)(\phi(x_0) - \phi^*).    
	$$
	For this particular case, a similar result, obtained in \cite{JMLR:v17:15-084} for strongly convex functions, namely
	$$
	\phi\big(x(t)\big)-\phi^*\le \left( 1 - \frac{3}{25} \left( \frac{67}{71}\right)^2\dfrac{\mu}{L}\right)(\phi(x_0) - \phi^*).  
	$$
	Our constant is approximately 66.37\% larger than the one from \cite{JMLR:v17:15-084}, which implies a greater reduction in the function values each time the trajectory is restarted. On the other hand, if $\beta>0$, we can still obtain a slightly smaller lower bound, namely
	$\Psi(\tau_3)>\left(\dfrac{2\alpha+1}{2\alpha+2}\right)^2$, independent from $\beta$ and $L$. The proof is technical and will be omitted.
\end{remark}


\subsubsection*{Proof of Theorem \ref{teo:convergence_rate}}

Adopt the notation in Definition \ref{def:restarted}, take any $\tau_*\in(0,\tau_2)\cap(0,T_*]$, and set
$$\tau^*=\tau_* \exp \left[ \dfrac{(\alpha+1)^2}{2\alpha\mu\tau_*^2\Psi(\tau_*)} \right], \qquad\hbox{where}\qquad
\Psi(\tau_*)=\left[2-\frac{1}{H(\tau_*)}\right]^2.$$
In view of Corollaries \ref{C:inf} and \ref{C:sup}, we have
$$\tau_*\le T(x_i)\le \tau^*$$
for all $i\ge 0$ (we assume $x_i\notin\argmin(\phi)$ since the result is trivial otherwise). Given $t>0$, let $m$ be the largest positive integer such that $m\tau^*\le t$. By time $t$, the trajectory will have been restarted at least $m$ times. By Proposition \ref{P:nonincreasing}, we know that
$$\phi\big(\chi_{x_0}(t)\big)\le \phi\big(\chi_{x_0}(m\tau^*)\big) \le \phi\big(\chi_{x_0}(m\tau_*)\big).$$
We may not apply Proposition \ref{P:reduction} repeatedly to deduce that
$$\phi\big(\chi_{x_0}(t)\big)-\phi^* \le Q^{m}\big(\phi(x_0) - \phi^*\big)\qquad\hbox{where}\qquad Q=\left[1-\dfrac{\alpha\mu\tau_*^2\Psi(\tau_*)}{(\alpha+1)^2}\right]<1.$$
By definition, $(m+1)\tau^*>t$, which entails $m>\frac{t}{\tau^*}-1$. Since $Q\in(0,1)$, we have
$$Q^m \le Q^{\frac{t}{\tau^*}-1}=\dfrac{1}{Q}\exp\left(\dfrac{\ln(Q)}{\tau^*}\,t\right),$$
and the result is established, with $C=Q^{-1}$ and $K=-\frac{1}{\tau^*}\ln(Q)$. The proof is finished due to the fact that $\phi(u)\le\phi^*+\frac{L}{2}\|u-u^*\|^2$ for every $u^*\in\argmin(\phi)$.\hfill$\square$\bigskip \\

The convergence rate given in Theorem \ref{teo:convergence_rate}, holds for $C$ and $K$ of the form
$$C=C(\tau_*)=\left[1-\dfrac{\alpha\mu\tau_*^2\Psi(\tau_*)}{(\alpha+1)^2}\right]^{-1}$$
and
$$K=K(\tau_*)=-\frac{1}{\tau_*}\exp \left[ -\dfrac{(\alpha+1)^2}{2\alpha\mu\tau_*^2\Psi(\tau_*)} \right]\ln\left[1-\dfrac{\alpha\mu\tau_*^2\Psi(\tau_*)}{(\alpha+1)^2}\right]> \dfrac{\alpha\mu\tau_*\Psi(\tau_*)}{(\alpha+1)^2}\exp \left[ -\dfrac{(\alpha+1)^2}{2\alpha\mu\tau_*^2\Psi(\tau_*)} \right],$$
for any $\tau_*\in(0,\tau_2)\cap(0,T_*]$. In view of \eqref{E:taus} and Corollary \ref{C:inf}, $\tau_*=\tau_3$ is a valid choice. On the other hand, t    he function $K(\cdot)$ vanishes at $\tau\in\{0,\tau_2\}$ and is positive on $(0,\tau_2)$. By continuity, it attains its maximum at some $\hat\tau_*\in(0,\tau_2)\cap(0,T_*]$. Therefore, $K(\hat \tau_*)$ yields the fastest convergence rate prediction in this framework. 

\begin{remark}
	It is possible to implement a fixed restart scheme. To this end, we modify Definition \ref{def:restarted} by setting $T_i\equiv \tau$, with any $\tau\in(0,\tau_2)\cap(0,T_*]$, such as $\hat\tau_*$ or $\tau_3$, for example. In theory, $\hat\tau_*$ gives the same convergence rate as the original restart scheme presented throughout this work. From a practical perspective, though, restarting the dynamics too soon may result in a poorer performance. Therefore, finding larger values of $\hat\tau_*$ and $\tau_3$ is crucial to implement a fixed restart (see Remarks \ref{R:tau_3} and \ref{r:bound}).
\end{remark}

\if{
	\begin{remark} The constants $K_1$ and $K_2$ are expressed in terms of other constants, to simplify the notation. Next, a detailed explanation of both constants is provided. 
		\begin{itemize}
			\item First, $K_1$ depends on the constant $C_1$ provided in Proposition \ref{P:reduction}, which is defined as 
			$$C_1=\dfrac{\alpha}{2(\alpha+1)^2}C_0 T_*^2 \mu,$$
			where $C_0=(10/11)^2$, and $T_*$ is provided on Lemma \ref{lem:lemma2}. 
			Notice that if $\beta=0$, we obtain that $K_1 = \exp(\tilde{C} \mu / L)$, which is a similar result to the one obtained by Su, Boyd and Candés in \cite{JMLR:v17:15-084}. 
			\item The constant $K_2$ depends on $C_2$, used in the proof of Lemma \ref{lem:lemma4}, which is 
			$$C_2 = \dfrac{(\alpha+1)^2}{\alpha C_0 T_*^2 }.$$
			Using that, 
			$$K_2 = \dfrac{\alpha}{2(\alpha+1)^2}\cdot \left( \dfrac{10}{11}\right)^2 T_* \mu \exp \left( -\dfrac{(\alpha+1)^2}{\alpha C_0 T_*^2 \mu}  \right) $$
		\end{itemize}
		From the definition of $T_*$, 
		$$T_*=-\rho_{\alpha}\beta + \sqrt{\rho_{\alpha}^2 \beta^2 + \dfrac{3}{\alpha L}} = \dfrac{3}{\alpha L\left( \sqrt{\rho_{\alpha}^2 \beta^2 + \frac{3}{\alpha L}}  + \rho_{\alpha} \beta \right)},$$
		with $\rho_{\alpha}=(\alpha+3)/(\alpha+2)$. Notice that the value of $T_*$ is decreasing in $\beta \geq 0$. Then, the value of $K_2$ decreases fast on $\beta$, so this rate of convergence is faster when $\beta=0$. It can be seen, numerically, that the upper bound on Lemma \ref{lem:lemma4} is not tight, so the authors have the intuition that the constant $K_2$ can be significantly improved. 
	\end{remark}
}\fi 

	\section{Numerical illustration} \label{sec:illustration}

In this section, we provide a very simple numerical example to illustrate how the convergence is improved by the restarting scheme. A more thorough numerical analysis will be carried out in a forthcoming paper, where implementable optimization algorithms will be analyzed. \\


\subsection{Example \ref{EG:AVD_GIN_R3} revisited}\label{sec:example_revisited}

We consider the quadratic function $\phi:\R^3\to\R$, defined in Example \ref{EG:AVD_R3} by \eqref{eq:Phi_R3}, with $\rho = 10$. We set $\alpha=3.1$ and $\beta=0.25$, and compute the solutions of \eqref{eq:avd_intr} and \eqref{eq:din_avd}, starting from  $x(1)=x_1=(1,1,1)$ and zero initial velocity, with and without restarting, using the \texttt{Python} tool \texttt{odeint} from the \texttt{scipy} package. Figure \ref{fig:cont} shows a comparison of the values along the trajectory with and without restarting, first for \eqref{eq:avd_intr}, and then for \eqref{eq:din_avd}. In both cases, the restarted trajectories appear to be more stable and converge faster. 



\begin{figure}[h]
	\centering
	{\includegraphics[width=0.4\textwidth]{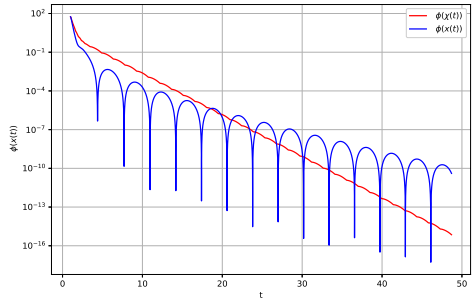}
	}
	\caption{Values along the trajectory, with (red) and without (blue) restarting, for 
		\eqref{eq:din_avd}
		.}
	\label{fig:cont}
\end{figure}


However, one can do better. As mentioned earlier, restarting schemes based on function values are effective from a practical perspective, but show an erratic behavior as the trajectory approaches a minimizer. It seems natural as a heuristic to use the first (or $n$-th) function-value restart point as a warm start, and then apply speed restarts, for which we have obtained convergence rate guarantees. Although the velocity must be set to zero {\it after} each restart, there are no constraints on the initial velocity used to compute the warm starting point. The results are shown in Figure \ref{fig:cont4}, with initial velocity set to zero and $\dot{x}(1)=-\beta\nabla\phi(x_1)$, respectively.\\


\begin{figure}[ht]
	\centering
	{\includegraphics[width=0.4\textwidth]{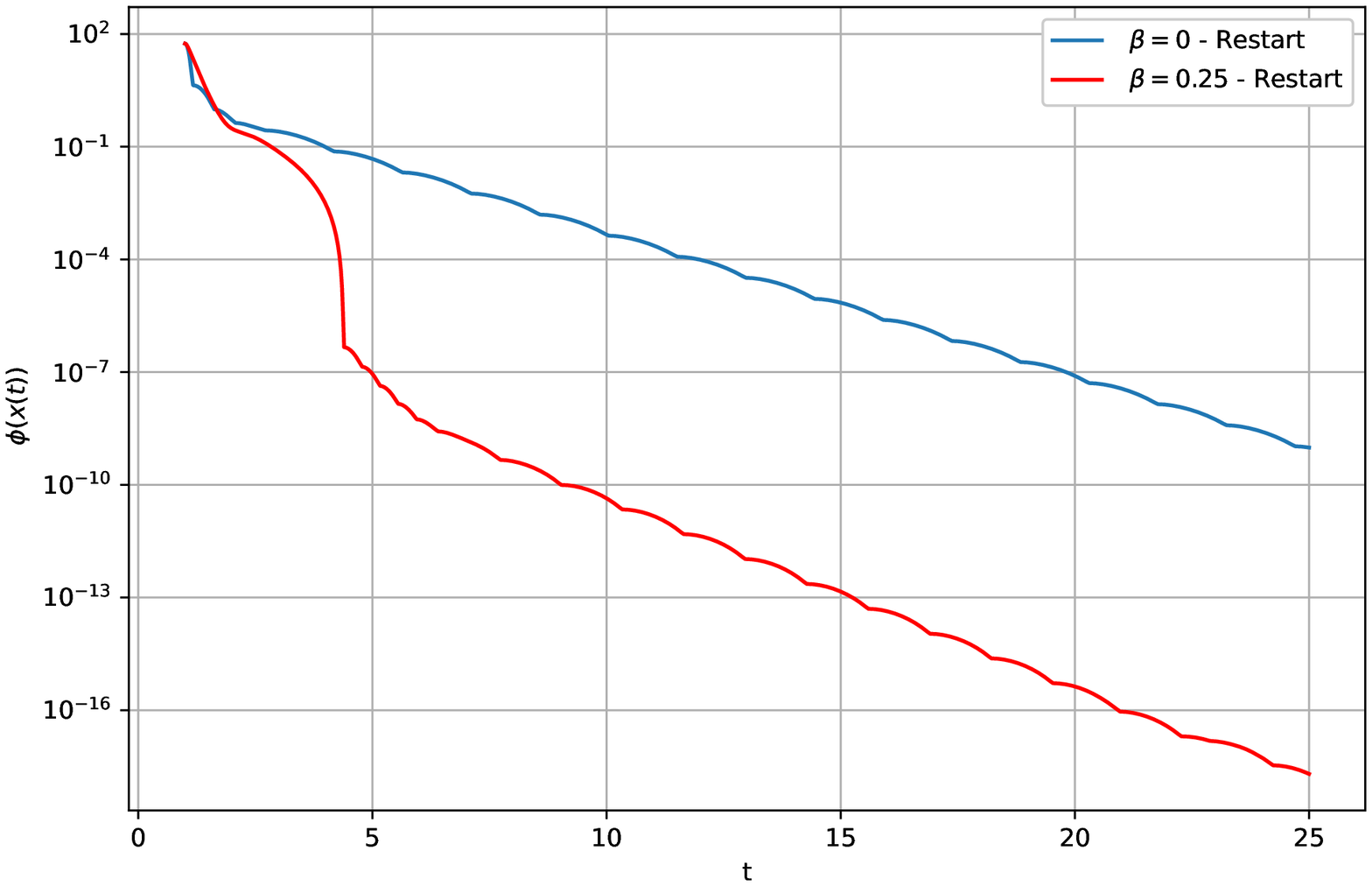}}
	{\includegraphics[width=0.4\textwidth]{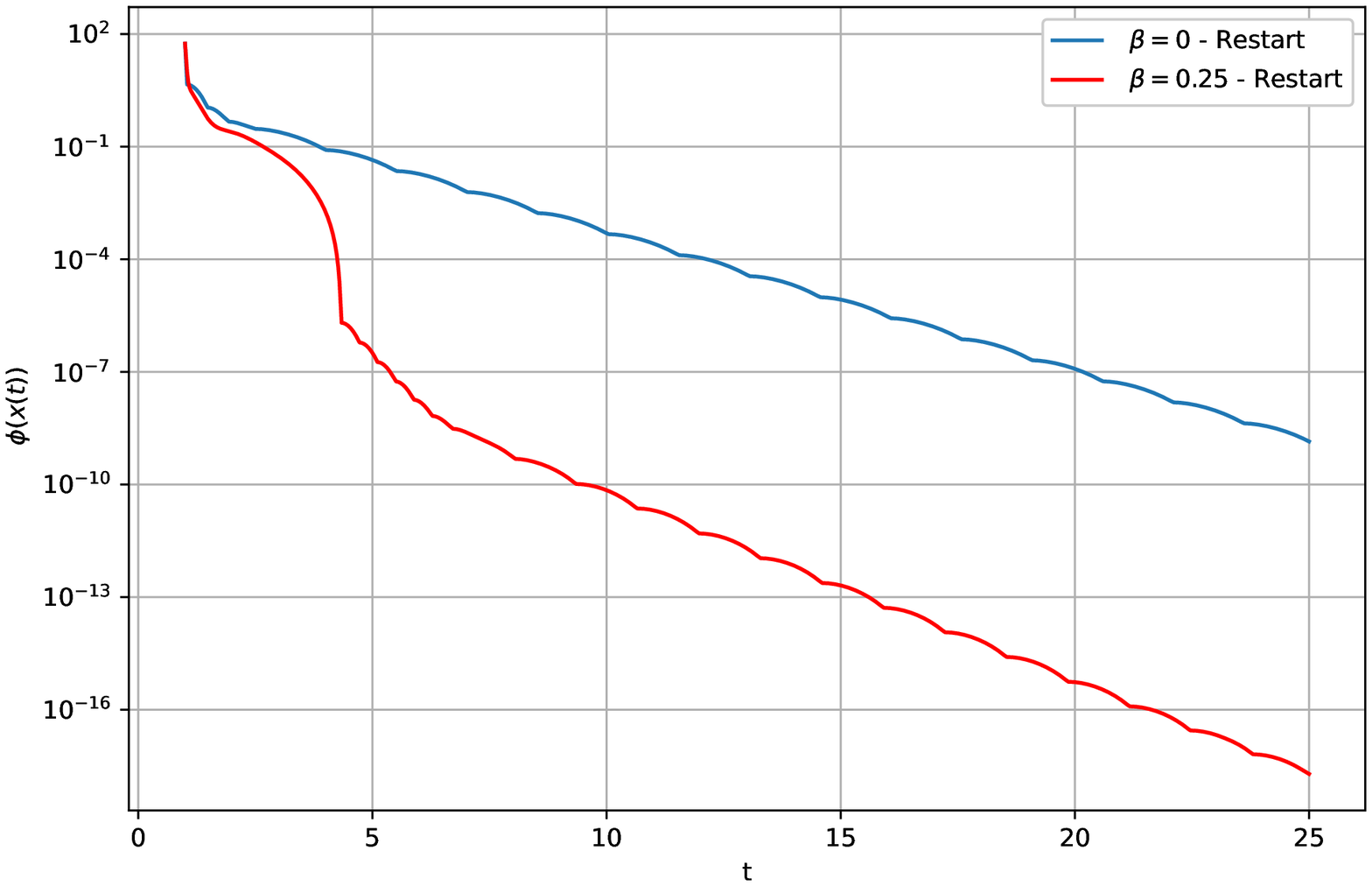}}
	{\includegraphics[width=0.4\textwidth]{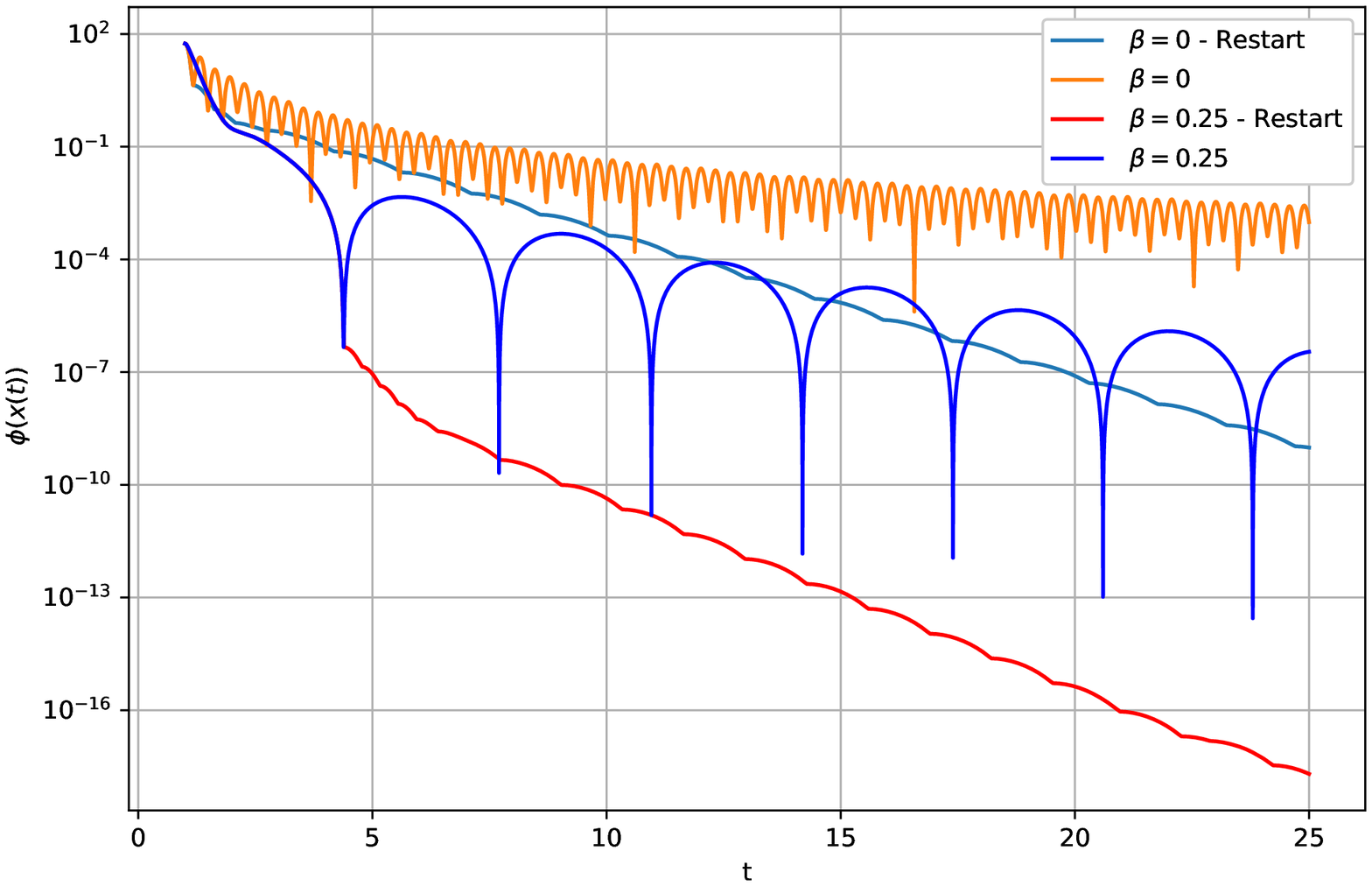}}
	{\includegraphics[width=0.4\textwidth]{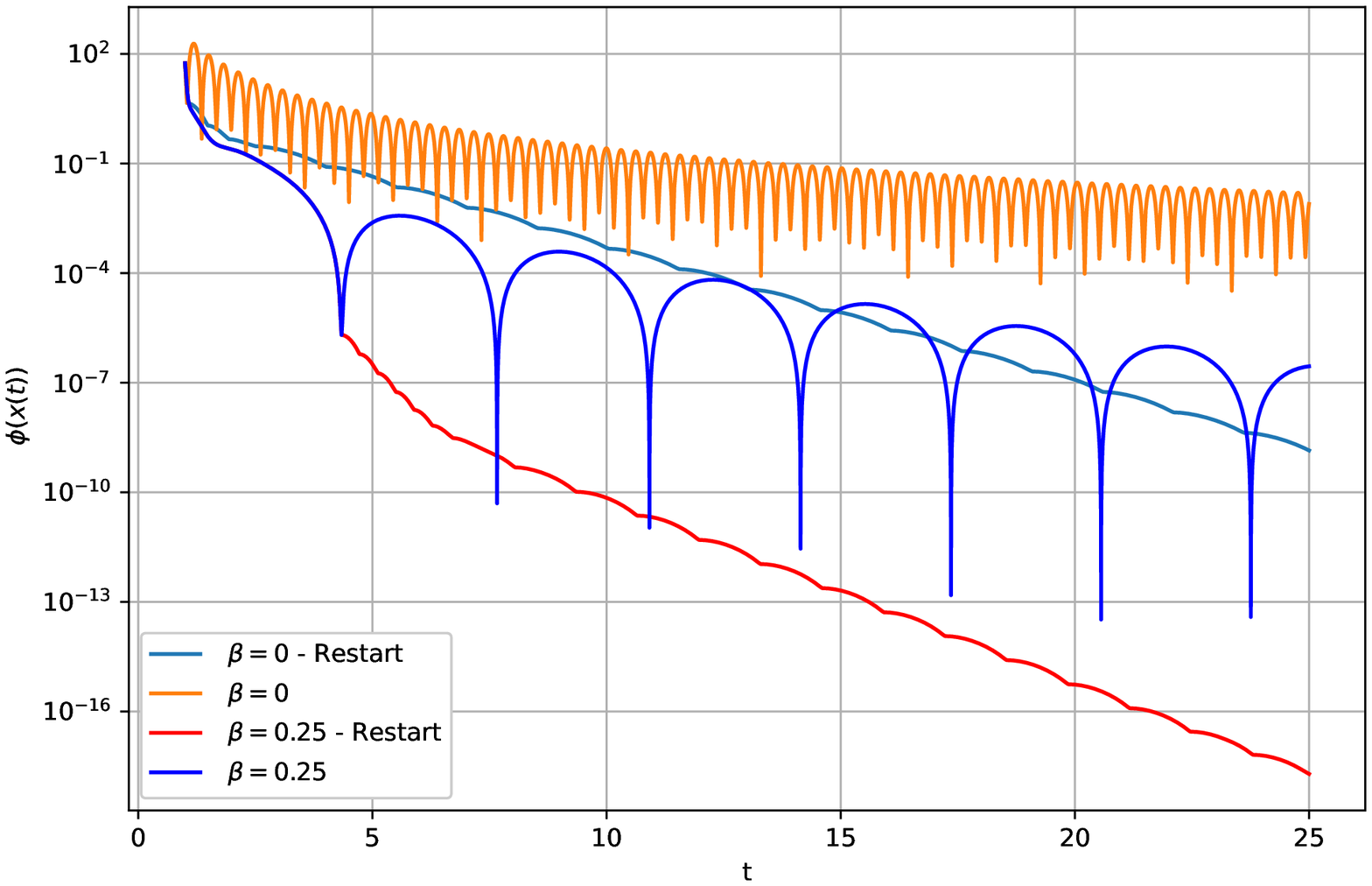}}
	\caption{Top: Values along the trajectory, with warm start, for \eqref{eq:avd_intr} (blue) and \eqref{eq:din_avd} (red), with inicial velocity set to zero (left) and $\dot{x}(1)=-\beta\nabla\phi(x_1)$ (right). Bottom: Includes trajectories without restarting, for reference.}
	\label{fig:cont4}
\end{figure}


A linear regression after the first restart provides estimations for the linear convergence rate of the function values along the corresponding trajectories, when modeled as $\phi\big(\chi(t)\big)\sim Ae^{-Bt}$, with $A,B>0$. The results are displayed in Table \ref{tab:coef_regression}. The absolute value of the exponent $B$ in the linear convergence rate is increased by 34,67\% in the case $\dot{x}(1)=0$, and by 39,86\% in the case $\dot{x}(1)=-\beta \nabla \phi (x_1)$. Also, the minimum values for the methods presented in Figure \ref{fig:cont4} can be analyzed. The last and best function values on $[1,25]$ are displayed on Table \ref{tab:values_t25}. In all cases, the best value without restart is approximately $10^4$ times larger than the one obtained with our policy. We also observe similar final values for the restarted trajectories despite the different initial velocities. 

\begin{table}[h]
	\centering
	\begin{tabular}{ccc|cc}
		\toprule
		&\multicolumn{2}{c}{$\dot{x}(1)=0$} &\multicolumn{2}{c}{$\dot{x}(1)=-\beta \nabla \phi (x_1)$} \\
		\midrule
		& $\beta=0$ & $\beta=0.25$ & $\beta=0$ & $\beta=0.25$\\
		\midrule
		$A$ & 3.7545 & 8.16e-6 & 3.2051 & 1.65e-05\\
		$B$ & 0.8837 & 1.1901 & 0.859 & 1.2014 \\
		\bottomrule
	\end{tabular}
	\caption{Coefficients in the linear regression, when approximating $\phi\big(\chi(t)\big)\sim Ae^{-Bt}$.}
	\label{tab:coef_regression}
\end{table}

\begin{table}[ht]
	\centering
	\begin{tabular}{rcc|cc}
		\toprule
		&\multicolumn{2}{c}{$\dot{x}(1)=0$} &\multicolumn{2}{c}{$\dot{x}(1)=-\beta \nabla \phi (x_1)$} \\
		\midrule
		& $\beta=0$ & $\beta=0.25$ & $\beta=0$ & $\beta=0.25$\\
		\midrule
		Last value without restart & 0.0009 & 3.4793e-07 & 0.0079 & 2.8094e-07 \\
		Best value without restart & 4.0697e-06 & 2.8024e-14 & 3.2770e-05 & 3.2760e-14 \\
		Last/best value with restart and warm start &9.8118e-10 & 2.0103e-18 &1.3940e-09 & 1.9452e-18  \\
		\bottomrule
	\end{tabular}
	\caption{Values reached for $\phi$ at $t=25$.}
	\label{tab:values_t25}
\end{table}


	

\subsection{A first exploration of the algorithmic consequences}

Different discretizations of \eqref{eq:din_avd} can be used to design implementable algorithms and generate minimizing sequences for $\phi$, which hopefully will share the stable behavior we observe in the solutions of \eqref{eq:din_avd}. Three such algorithms were first proposed in \cite{attouch2020first}, for which we implemented a speed restart scheme, analogue to the one we have used for the solutions of \eqref{eq:din_avd}. Since we obtained very similar results and the numerical analysis of algorithms is not the focus of this paper, we describe only the simplest one in detail, and present the numerical results for that one. As in \cite{JMLR:v17:15-084}, a parameter $k_{\min}$ is introduced, to avoid two consecutive restarts to be too close.


\begin{algorithm}[h]
	\SetAlgoLined
	Choose $x_0$, $x_1 \in \R^n$, $N$, $k_{\min}$ and $h>0$. \\
	\For{$k=1 \ldots N$}{
		\text{Compute} $y_k= x_k + (1-\frac{\alpha}{k})(x_k-x_{k-1}) - \beta h(\nabla \phi(x_k) - \nabla \phi(x_{k-1}))$, \\
		\text{and then} $x_{k+1} = y_k - h^2 \nabla \phi(y_k)$. \\
		\uIf{$\norm{x_{k+1} - x_{k}}<\norm{x_k - x_{k-1}}$ {\bf and} $k\geq k_{\min}$}{k=1;}
		\uElse{k=k+1.}
	}
	\Return $x_N$.
	\caption{Inertial Gradient Algorithm with Hessian Damping (IGAHD) - Speed Restart version}
	\label{algorithm:restart}
\end{algorithm}


	
	\begin{example} \label{EG:algo1}
		We begin by applying Algorithm \ref{algorithm:restart}, as well as the variation with the warm start, to the function $\phi:\R^3 \mapsto \R $ in Examples \ref{EG:AVD_R3} and \ref{EG:AVD_GIN_R3}, with the parameters $k_{\min}=10$, $\beta=h=1/\sqrt{L}$ and $\alpha=3.1$. Figure \ref{fig:discrete2} shows the evolution of the function values along the iterations. The coefficients in the approximation $\phi(x_k) \sim Ae^{-Bt}$, with $A,B >0$, obtained for each algorithm, are detailed on Table \ref{tab:coefs_discrete_ill}. As one would expect, the value of $B$ is similar and that of $A$ is significantly lower. Also, Table \ref{tab:algo1} shows the values obtained along 1000 iterations. The best value without restart is $10^5$ times larger than the one obtained with our policy.

		\begin{figure}[ht]
			\centering
			{\includegraphics[width=0.4\textwidth]{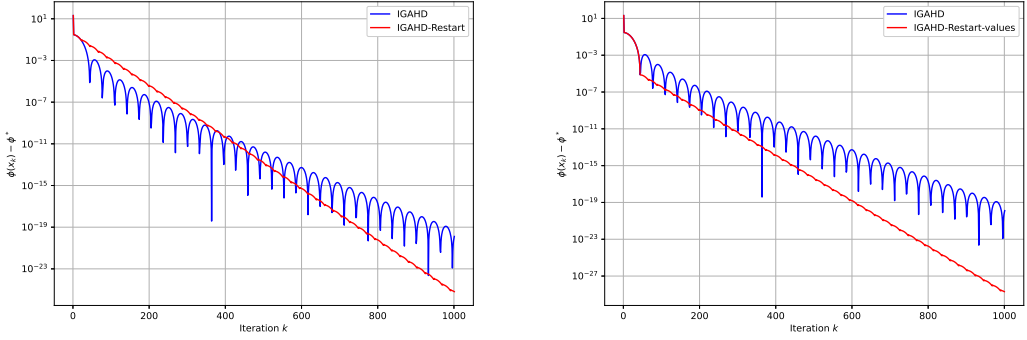}}
			{\includegraphics[width=0.4\textwidth]{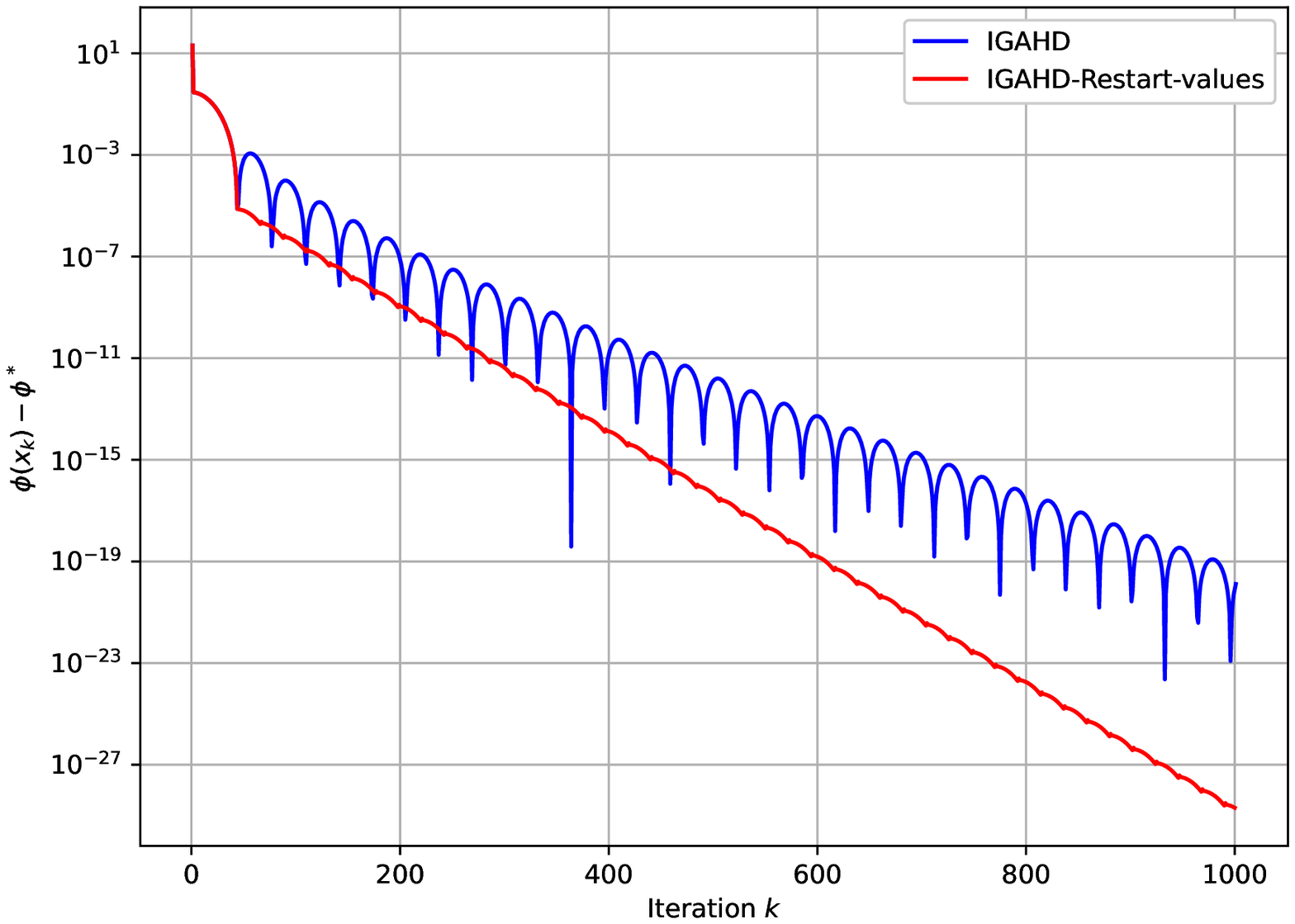}}
			
			\caption{Function values along iterations of Algorithm \ref{algorithm:restart} without (left) and with (right) warm start.}
			\label{fig:discrete2}
		\end{figure}

		\begin{table}[ht]
			\centering
			\begin{tabular}{c|c|c}
				\toprule
				& Algorithm \ref{algorithm:restart}  & Algorithm \ref{algorithm:restart} with warm start  \\
				\midrule
				$A$ & 0.3722 & 1.0749e-4 \\
				$B$ & 0.0571 & 0.057 \\
				\bottomrule
			\end{tabular}
			\caption{Coefficients in the linear regression for Example \ref{EG:algo1}.}
			\label{tab:coefs_discrete_ill}
		\end{table}

		\begin{table}[ht]
			\centering
			\begin{tabular}{rcr}
				\toprule
				Last iteration without restart & \quad &  1.2927e-20 \\
				Best iteration without restart & & 2.2907e-24 \\
				Last/best iteration with restart and warm start & &  2.0206e-29 \\
				\bottomrule 
			\end{tabular}
			\caption{Functions values for Example \ref{EG:algo1}.}
			\label{tab:algo1}
		\end{table}

		

	\end{example}

	

	\begin{example} \label{EG:algo2}
		Given a positive definite symmetric matrix $A$ of size $n\times n$, and a vector $b \in \R^n$, define $\phi:\R^n \mapsto \R$ by
		$$\phi(x)=\dfrac{1}{2}x^T A x + b^Tx.$$
		
		For the experiment, $n=500$, $A$ is randomly generated with eigenvalues in $(0,1)$, and $b$ is also chosen at random. We first compute $L$, and set $k_{\min}=10$, $h=1/\sqrt{L}$, $\alpha=3.1$ and $\beta=h$. The initial points $x_0=x_1$ are generated randomly as well. Figure \ref{fig:discrete1} shows the comparison for Algorithm \ref{algorithm:restart} and a variation of it giving a warm start as the one described in the continuous setting. That is, to restart the first time when the function increases instead of decrease, and then performing the speed restart detailed on Algorithm \ref{algorithm:restart}. It can be seen, that the restart scheme stabilizes and accelerates the convergence in both cases. The coefficients obtained for each algorithm in the approximation $\phi(x_k) \sim Ae^{-Bt}$, with $A,B >0$, are presented in Table \ref{tab:coefs_discrete_quad}. Also, Table \ref{tab:algo2} shows the value gaps obtained along 1800 iterations. The best value without restart is more than $10^4$ times larger than the one obtained with restart.
		
		
		\begin{figure}[ht]
			\centering
			{\includegraphics[width=0.4\textwidth]{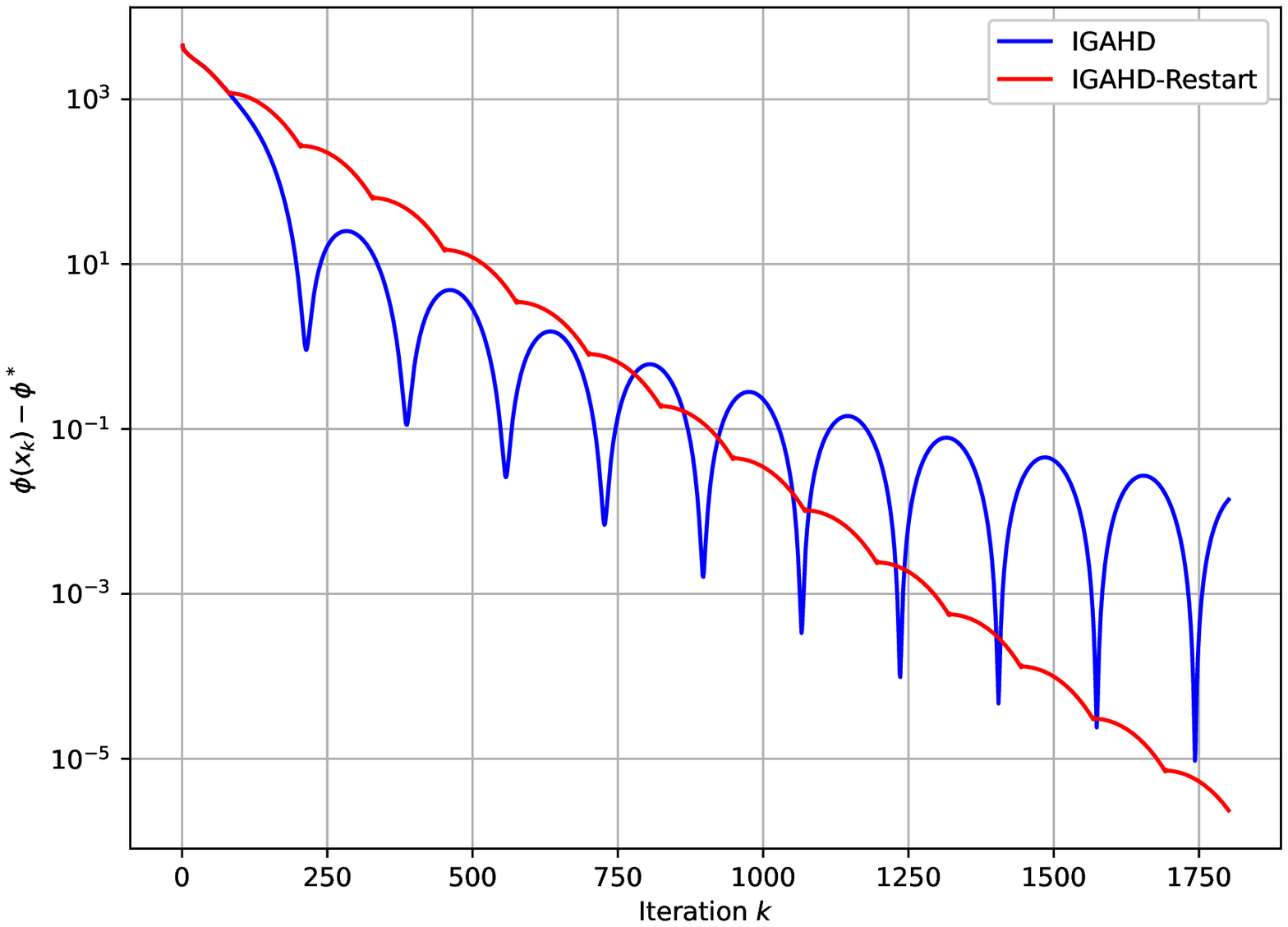}}
			{\includegraphics[width=0.4\textwidth]{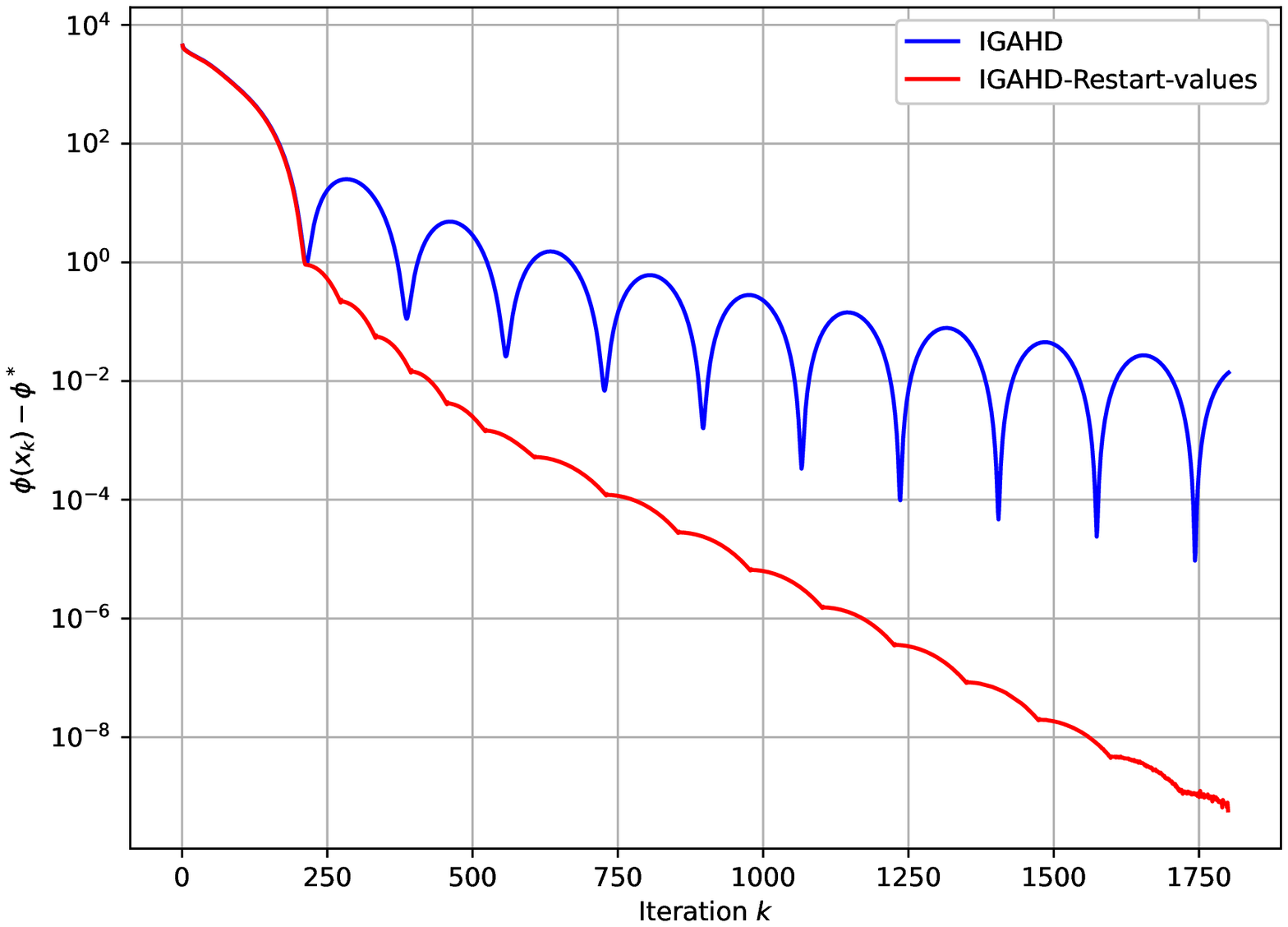}}
			
			\caption{Function values along iterations of Algorithm \ref{algorithm:restart} without (left) and with (right) warm start.}
			\label{fig:discrete1}
		\end{figure}
		
		\begin{table}[ht]
			\centering
			\begin{tabular}{c|c|c}
				\toprule
				& Algorithm \ref{algorithm:restart}  & Algorithm \ref{algorithm:restart} with warm start  \\
				\midrule
				$A$ & 3813.01& 1.6142 \\
				$B$ & 0.0117 & 0.0121 \\
				\bottomrule
			\end{tabular}
			\caption{Coefficients in the linear regression for Example \ref{EG:algo2}.}
			\label{tab:coefs_discrete_quad}
		\end{table}
		
		\begin{table}[ht]
			\centering
			\begin{tabular}{rcr}
				\toprule
				Last iteration without restart & \quad &  0.0139 \\
				Best iteration without restart & & 9.4293e-06 \\
				Last/best iteration with restart and warm start & &  5.8481e-10 \\
				\bottomrule 
			\end{tabular}
			\caption{Function values for Example \ref{EG:algo2}.}
			\label{tab:algo2}
		\end{table}
		
		

	\end{example}

	\appendix 
	\section{Appendix: Proof of Theorem \ref{theo:existence}} \label{Appendix}

Consider the differential equation
\begin{equation}\label{eq:edo_general2}
	\ddot{x}(t) + \gamma(t)\dot{x}(t) + \mathcal F\big(x(t)\big)\dot{x}(t)+\mathcal G\big(x(t)\big)=0.
\end{equation}
We assume that $\gamma$ is continuous and positive, with $\lim_{t\to0}\gamma(t)=+\infty$, and that $\cal F$ and $\cal G$ are (continuous and) sufficiently regular so that the differential equation \eqref{eq:edo_general2}, with initial condition $x(\delta)=x_\delta$ and $\dot x(\delta)=v_\delta$, has a unique solution defined on $[\delta,T_\infty)$ for some $T_\infty\in(0,\infty]$ and all $\delta>0$. Let
\begin{equation}\label{eq:def_M2}
	M(\delta,t):=\sup_{s \in [\delta,t]}\big\{ \gamma(s)\norm{\dot{x}_\delta(s)-v_0} \big\}.
\end{equation}

We have the following:

\begin{theorem}\label{theo:general2} Assume there is $T>0$ such that 
	\begin{equation}\label{hip:bound_M2}
		\sup_{0<\delta\le t\le T}M(\delta,t) < +\infty.
	\end{equation}
	Then, the differential equation \eqref{eq:edo_general2}, with initial condition $x(0)=x_0$ and $\dot x(0)=v_0$, has a solution. 
\end{theorem}

\begin{proof}
	For $\delta\in(0,T)$, define $x_\delta:[0,T]\to\R^n$ as follows: for $t\in[0,\delta]$, $x_\delta(t)=x_0+tv_0$; and for $t>\delta$, $x_\delta$ is the solution of \eqref{eq:edo_general2} with initial condition $x(\delta)=x_0+\delta v_0$ and $\dot x(\delta)=v_0$. Notice that $x_{\delta}$ is a continuous function such that matches a solution of \eqref{eq:edo_general2} on $[\delta,T]$ . From the hypotheses, there exist $c,K>0$ and such that $\gamma(t)\ge c$ and $M(\delta,t)\le K$ for all $0<\delta\le t\le T$. Therefore,
	$$c\,\|\dot x_\delta(s)-v_0\|\le \gamma(s)\|\dot x_\delta(s)-v_0\|\le M(\delta,t)\le K$$
	whenever $0<\delta\le s\le t\le T$, so that
	$$\|\dot x_\delta(s)-v_0\|\le \frac{K}{c},$$
	for all $s\in[0,T]$. As a consequence,
	$$\|x_\delta(s)-x_0\|\le \int_0^s\|\dot x_\delta(\tau)\|\,d\tau \le \|v_0\|\delta+\frac{KT}{c}$$
	on $[0,T]$. It follows that $(x_\delta)$ is bounded in $H^1(0,T;\R^n)$.
	By weak sequential compactness and the Rellich–Kondrachov Theorem (see, for instance \cite[Theorem 9.16]{brezis_functional_2011}), there is a sequence $(\delta_n)$ converging to zero, such that $x_{\delta_n}$ converges uniformly to a continuous function $x^*$, while $\dot x_{\delta_n}$ converges weakly in $L^2(0,T;\R^n)$ to some $y^*$. \\
	
	Clearly, $x^*(0)=x_0$. In turn, for $t\in(0,T]$, by the Mean Value Theorem and the definition of $M$, we have
	$$\left\|\frac{x^*(t)-x_0}{t}-v_0\right\|=\lim_{n\to\infty}\left\|\frac{x_{\delta_n}(t)-x_0}{t}-v_0\right\|=\lim_{n\to\infty}\|\dot x_{\delta_n}(c_n)-v_0\|\le \frac{\bar K}{\min\limits_{s\in(0,t]}\gamma(s)},$$
	which tends to zero as $t\to 0$. It remains to prove that $x^*$ satisfies \eqref{eq:edo_general2}. To this end, take any $t_0\in(0,T)$, and observe that $\delta_n<t_0$ for all sufficiently large $n$. Therefore, $x_{\delta_n}$ satisfies \eqref{eq:edo_general2} on $[t_0,T)$ for all such $n$. Multiplying by
	$$\Gamma(t):=\exp\left(\int_{t_0}^t\gamma(s)\,ds\right),$$
	we deduce that
	$$\Gamma(t)\dot x_{\delta_n}(t)-\Gamma(t_0)\dot x_{\delta_n}(t_0)+\int_{t_0}^t\Gamma(s)\,\mathcal F\big(x_{\delta_n}(s)\big)\dot x_{\delta_n}(s)\,ds+\int_{t_0}^t\Gamma(s)\,\mathcal G\big(x_{\delta_n}(s)\big)\,ds=0.$$
	By taking yet another subsequence if necessary, we may assume that $\dot x_{\delta_n}(t_0)$ converges to some $v^*$. From the uniform convergence of $x_{\delta_n}$ to $x^*$ on $[0,T]$, and the weak convergence of $\dot x_{\delta_n}$ to $y^*$ in $L^2(0,T;\R^n)$, it ensues that
	$$\Gamma(t)y^*(t)-\Gamma(t_0)\dot v^*+\int_{t_0}^t\Gamma(s)\,\mathcal F\big(x^*(s)\big)y^*(s)\,ds+\int_{t_0}^t\Gamma(s)\,\mathcal G\big(x^*(s)\big)\,ds=0$$
	for all $t\in (t_0,T)$. As a consequence, $x^*$ is continuously differentiable, $\dot x^*=y$, and $x^*$ satisfies \eqref{eq:edo_general2}.
\end{proof}

\begin{corollary}
	Equation \eqref{eq:din_avd} has at least one solution.
\end{corollary}

\begin{proof}
	According to Theorem \ref{theo:general2}, for the existence, it suffices to show that the expression $M(\delta,t)$, defined in \eqref{eq:def_M2}, is bounded for $0<\delta\le t\le T$, for some $T>0$. Mimicking the proof of Lemma \ref{lem:lemma1}, we show that
	$$H(t)M(\delta,t)\le \dfrac{\norm{\nabla \phi(x_0)}}{\alpha +1},\qquad\hbox{with}\qquad H(t)=1-\frac{\beta Lt}{\alpha+2}-\frac{Lt^2}{2(\alpha+3)}.$$
	The only positive zero of $H$ is $\tau_1$, given by \eqref{E:tau1}, and $H$ is decreasing on $(0,\tau_1)$. Hence, if $T<\tau_1$, then
	$$\sup_{0<\delta\le t\le T}M(\delta,t) \le \dfrac{\norm{\nabla \phi(x_0)}}{(\alpha +1)H(T)} < +\infty,$$
	as claimed.
\end{proof}

\begin{proposition}
	Equation \eqref{eq:din_avd}, with initial condition $x(0)=x_0$ and $\dot x(0)=0$, has at most one solution in a neighborhood of $t=0$.
\end{proposition}

\begin{proof}
	Let $x$ and $y$ satisfy \eqref{eq:din_avd} with the same initial state and null initial velocity. We define 
	$$\tilde{M}(t) = \sup_{u \in [0,t)}\lbrace \norm{\dot{x}(u)-\dot{y}(u)} \rbrace,$$
	and proceed as in the proof of Lemma \ref{L:I_and_J}, to obtain
	\begin{equation}\label{eq:gradient_mtilde}
		\norm{\nabla \phi(x(t)) - \nabla \phi(y(t))} 
		\leq L t \tilde{M}(t) 
	\end{equation}
	
	As $x$ and $y$ satisfy \eqref{eq:din_avd}, we  integrate by parts to obtain
	\begin{align*}
		t^\alpha (\dot{x}(t) - \dot{y}(t)) &= - \int_{0}^{t}u^\alpha \left( \nabla \phi(x(u)) - \nabla \phi(y(u)) \right) \, du  - \beta \int_{0}^{t}u^\alpha\left( \nabla^2 \phi(x(u))\dot{x}(u) - \nabla^2 \phi(y(u))\dot{y}(u) \right) \,du \\
		&= - \int_{0}^{t}u^\alpha \left( \nabla \phi(x(u)) - \nabla \phi(y(u)) \right) \, du  - \beta \int_{0}^{t}u^\alpha \dfrac{d}{du}\left( \nabla \phi(x(u)) - \nabla \phi(y(u)) \right)  \,du \\
		&= - \int_{0}^{t}u^\alpha \left( \nabla \phi(x(u)) - \nabla \phi(y(u)) \right) \, du - \beta t^\alpha \left( \nabla \phi(x(t)) - \nabla \phi(y(t)) \right) \\
		&\quad + \alpha \int_{0}^{t}u^{\alpha-1} \left( \nabla \phi(x(u)) - \nabla \phi(y(u)) \right) \, du.
	\end{align*}
	Using \eqref{eq:gradient_mtilde}, and the fact that $\tilde{M}(t)$ is increasing, we get 
	\begin{align*}
		t^\alpha \norm{\dot{x}(t) - \dot{y}(t)} &\leq \int_{0}^{t}L u^{\alpha+1} \tilde{M}(u) \, du + \beta L t^{\alpha+1} \tilde{M}(t) + \alpha\beta \int_{0}^{t}L u^\alpha \tilde{M}(u) \, du \\
		& \le \dfrac{1}{\alpha+2}L \tilde{M}(t) t^{\alpha+2} + \dfrac{2\alpha+1}{\alpha+1}\beta L \tilde{M}(t) t^{\alpha+1}.
	\end{align*}
	Then, 
	$$\norm{\dot{x}(t) - \dot{y}(t)} \leq \dfrac{1}{\alpha+2}L \tilde{M}(T) T^2 + \dfrac{2\alpha+1}{\alpha+1}\beta L \tilde{M}(T) T,$$
	whenever $0<t\le T$. Taking supremum, we conclude that
	$$Q(t)\tilde{M}(T)\le 0\qquad\hbox{with}\qquad Q(t)=1-\dfrac{2\alpha+1}{\alpha+1}\beta L t-\dfrac{1}{\alpha+2}Lt^2,$$
	for all $T>0$. Since $Q(T)>0$ in a neighborhood of $0$, it follows that $\tilde M$ must vanish there, whence $x$ and $y$ must coincide. 
\end{proof}
	\bibliographystyle{abbrv}
	\bibliography{referencias}
\end{document}